\documentclass[12pt]{amsart}
\usepackage{amssymb}

\usepackage{amsmath,amsthm,amsfonts,amssymb,cmap,enumerate,rotating,graphicx,mathrsfs,tikz}
\usetikzlibrary{matrix,arrows}
\usepackage{enumitem}
\usepackage{epsfig}\newif\ifAMS
\IfFileExists{amssymb.sty}
  {\AMStrue\usepackage{amssymb}}
  {\usepackage{latexsym}}
  \usepackage{enumerate}
\usepackage{amsthm}
\usepackage{amsmath,amscd}
\usepackage{verbatim}
\usepackage[all,cmtip]{xy}
\usepackage{fancyhdr}
\usepackage[utf8]{inputenc}
\usepackage{cite}
\usepackage{afterpage}
\usepackage{booktabs}

\usepackage{standalone}

\theoremstyle{plain}

\newtheorem{theorem}{Theorem}[section]
\newtheorem{lemma}[theorem]{Lemma}
\newtheorem{corollary}[theorem]{Corollary}
\newtheorem{proposition}[theorem]{Proposition}

\newtheorem{remark}[theorem]{Remark}
\newtheorem{example}[theorem]{Example}
\theoremstyle{definition}
\newtheorem{Def}[theorem]{Definition}

\newcommand {\N}{{\mathbb N}}

\newcommand{\C}{{\mathbb C}}

\title{Subgroups of Quantum Groups}
\author{Georgia Christodoulou} 
\address{Mathematical Institute, Andrew Wiles Building, University of Oxford, Radcliffe Observatory Quater, Woodstock Rd, Oxford OX2 6GG.}
\email{georgia.christodoul@gmail.com}

\begin{document}
\maketitle

\begin{abstract}
We investigate the notion of a subgroup of a quantum group. We suggest a general definition, which takes into account the work that has been done for quantum homogeneous spaces. We further restrict our attention to reductive subgroups, where some faithful flatness conditions apply. Furthermore, we proceed with a categorical approach to the problem of finding quantum subgroups. We translate all existing results into the language of module and monoidal categories and give another characterization of the notion of a quantum subgroup. 
\end{abstract}

\tableofcontents   


\section{Introduction}

Although quantum groups have been defined and understood for several years now, the notion of a quantum subgroup is still a bit unclear in the literature.  Consider a general reductive group $G$. We know that  both the universal eneveloping algebra $U(\mathfrak{g})$ and the coordinate ring $\mathcal{O}(G)$ have the structure of a Hopf algebra. The same is true for their quantum counterparts $U_q(\mathfrak{g})$ and $\mathcal{O}_q(G)$. Now, given a subgroup $M$ of $G$ we would like to define its quantum version $\mathcal{O}_q(M)$. The first guess is to require for this to give rise to a Hopf algebraic structure as well, and therefore to look at the ``correct" (with respect to $M$)  Hopf subalgebra of  $U_q(\mathfrak{g})$ or quotient Hopf agebra of $\mathcal{O}_q(G)$. This is what happens when $M$ is for instance a Borel subgroup of $G$.  The need for a Hopf algebra structure has also been demonstrated in several cases (see for example \cite{Schneider}, \cite{Parshall&Wang}, or \cite{S.Wang}) where by definition a quantum subgroup is first of all a Hopf algebra satisfying some extra conditions.
However, it has been clear for some time now that this definition is too restrictive in general. The theory that has been developed for quantum homogeneous spaces shows that Hopf subalgebras of a given Hopf algebra are not enough. Podl\'es in \cite{Podles} constructs a class of quantum homogeneous spaces with an $SU_q(2)$-action which correspond to the classical 2-sphere $SU(2)/SO(2)$. Dijkhuizen in \cite{Dijkhuizen} gives a survey regarding the construction of some compact quantum symmetric spaces such as $SU(n)/SO(n)$ or $SU(2n)/Sp(n)$. Letzter in \cite{Letzter} constructs quantum symmetric spaces. In all of the above cases, quantum homogeneous spaces are defined as coideal subalgebras and cannot be associated to any Hopf algebra. In addition to this, M\"uller and Schneider in \cite{Muller&Schneider} proved that in the cases of Podl\'es and  Dijkhuizen above, $\mathcal{O}_q[G]$ is faithfully flat over the defined coideal subalgebra, enriching at the same time the theory with the notions of semisimplicity and cosemisimplicity.

Let us see how the language of coideal subalgebras together with the faithful flatness condition can give us a definition for a quantum subgroup.
Given a closed subgroup $M$ of $G$ we can consider its corresponding Lie subalgebra $\mathfrak{m} \subset \mathfrak{g}$. Then $U(\mathfrak{m}) \subset U(\mathfrak{g})$ is a left coideal subalgebra. In the quantum case, if we can find a coideal of $U_q(\mathfrak{g})$ to be the quantum analogue of $U(\mathfrak{m}) $ then using the general theory of Hopf algebras we can construct a right coideal subalgebra $A$ of $\mathcal{O}_q(G)$ but also a quotient left coalgebra $B$. Moreover it is known that $\mathcal{O}_q(G)$ is faithfully flat over $A$ if and only if it is faithfully coflat over $B$. Moreover we can consider the category $\mathcal{M}^{\mathcal{O}_q(G)}_A$ whose objects are $\mathcal{O}_q(G)$-comodules which are also $A$-modules with the compatibility condition that the $\mathcal{O}_q(G)$-comodule map is a morphism of $A$-modules. Then, under the faithfully flat condition, $\mathcal{M}^{\mathcal{O}_q(G)}_A$ is equivalent to the category of $B$-comodules. We use this quotient coalgebra $B$ as the definition of the quantum subgroup corresponding to $M$. Our intuition is that if we think of $A$ as the quantum coordinate algebra $\mathcal{O}_q(G/M)$ corresponding to the classical homogeneous space $G/M$, then the category $\mathcal{M}^{\mathcal{O}_q(G)}_A$ corresponds to vector spaces with an action of the (unknown) quantum subgroup coorresponding to $M$. By the above, this is equivalent to the category of $B$-comodules. Hence, $B$ carries the right representation theory with respect to the classical $M$, and therefore can be used as a definition for $\mathcal{O}_q(M)$. This is one way to look at quantum subgroups. 

However, in addition to the above, there is an alternative, categorical approach to the problem of defining quantum subgroups. Classically, whenever $M$ is a subgroup of $G$ we can consider the categories of their representations. $\text{Rep}(G)$ is a monoidal category, and $\text{Rep}(M)$ becomes a module category over it. Moreover between the two categories can be defined a restriction and an induction functor. The same should be true for their quantum analogues.  The category of $\mathcal{O}_q(G)$-comodules is again a monoidal category, and it is natural to expect that the category of comodules corresponding to a quantum subgroup of $G$ should have the structure of a module category over it as happens in the classical case. So, in order to define a quantum subgroup it is enough to find the appropriate module category over the category of $\mathcal{O}_q(G)$-comodules. The idea that module categories are related to quantum subgroups has been found in previous works as well. In \cite{Kirillov+Ostrik} Kirillov and Ostrik classify the ``finite subgroups in $U_q(\mathfrak{sl}_2)$” , where $q = e^{\pi i / l}$ is a root of unity. For them a subgroup in  $U_q(\mathfrak{sl}_2)$ is a commutative associative algebra in a tensor category $\mathcal{C}$.  Similarly, in \cite{Grossman&Snyder} Grossman and Snyder interpret quantum subgroups of finite groups as simple module categories over the category $\mathcal{C}$ of $G$-modules. Ocneanu in \cite{Ocn} uses similar definitions to classify quantum subgroups of $SU(n)$. This approach gives a categorical characterization of the notion of a quantum subgroup and it provides a dictionary between the two worlds that gives us more flexibility when dealing with the theory of quantum subgroups.

Let us now say a few words on how this paper is organized. 
The first part deals with the definition of a quantum subgroup using the language of Hopf algebras. 
In section \ref{Hopf Algebras} we give the main results from the theory of Hopf algebras which we think are the correct tools for the definition of a quantum subgroup. These are results that have been known for some time now, but not exactly used in this context. We look at both coideal subalgebras and quotient coalgebras, which are dual to each other and provide two equivalent perspectives. We see how this correspondence is achieved and obtain results for categories of representations over them. In section \ref{ss and cs} we also discuss the notions of semisimplicity and cosemisimplicity and how they are related to the faithful flatness condition. Finally in \ref{def of qs} we give our definition. We suggest that quantum subgroups should correspond to certain module quotient coalgebras (or equivalently coideal subalgebras) of a Hopf algebra. In particular if we restrict to subgroups for which the corresponding quotient space is affine, a faithful coflatness condition (respectively faithful flatness condition) is required.

In the second part of the paper we focus on the categorical approach of the problem. We begin by recalling the general theory of monoidal and module categories. In section \ref{prwto} we formulate the problem of definining quantum subgroups using the language of module categories and give our alternative definitions. In particular, we start with a monoidal category $\mathcal{C}$ and a module category $\mathcal{M}$ over it. We prove in Theorem \ref{theorem 1} that if there exist adjoint functors of module categories $Res : \mathcal{C} \rightarrow \mathcal{M}$ and $Ind : \mathcal{M} \rightarrow \mathcal{C}$ such that $Ind$ is exact and faithful, then the module category $\mathcal{M}$ is equivalent to a module category $\text{Mod}_{\mathcal{C}}(A)$ for an algebra $A \in \mathcal{C}$. We investigate the properties of such module categories when $\mathcal{C}$ is taken to be the category of comodules over a Hopf algebra $H$. 
In section \ref{deytero} we proceed one step further. In our main theorem, Theorem \ref{theorem 2}, we prove that if $\mathcal{C}$ is $\mathcal{M}^H$, namely the category of right comodules for a Hopf algebra $H$ with bijective antipode, $\mathcal{M}$ is $\mathcal{M}^B$, the category of right comodules for a coalgebra $B$, and if $Res$ in the adjunction above is moreover a functor of module categories that carries the forgetful functor to the forgetful functor then $B$ has the structure of a quotient $H$-module coalgebra and $H$ is faithfully coflat over $B$. Furthermore there exists a coideal subalgebra $A'$ of $H$ such that $\mathcal{M} \simeq  \text{Mod}_{\mathcal{C}}(A')$ and $H$ is faithfully flat over $A'$. This gives us the categorical characterization of a quantum subgroup. We finally show in \ref{converse} that quantum subgroups satisfy the conditions of Theorem \ref{theorem 2}.\\ \\


\section{Hopf Algebras} \label{Hopf Algebras}
\subsection{Definitions and notation}

We recall that if $\mathfrak{g}$ is a Lie algebra, then its universal enveloping algebra $U(\mathfrak{g})$ is a Hopf algebra. The same is true for the quantized enveloping algebra $U_q(\mathfrak{g})$ corresponding to $\mathfrak{g}$. Moreover if $G$ is the connected, simply connected Lie group with
Lie algebra $\mathfrak{g}$ then for each type 1 finite--dimensional representation, we can define matrix coefficients and
define the quantized coordinate algebra  $\mathcal{O}_q[G]$ which is again a Hopf algebra. There is a natural pairing
$( , ) : U_q(\mathfrak{g}) \times \mathcal{O}_q[G] \rightarrow k$.

Throughout this section, we let $H$ be a Hopf algebra over an algebraically closed field $k$. We denote by $\Delta_H$ the comultiplication and by $\epsilon_H$ the counit of the coalgebraic structure. $S_H$ will denote the antipode and $1_H$ the unit. We also adopt the Sweedler notation, therefore $\Delta_H(h) = h_{(1)} \otimes h_{(2)}$. In the case of coactions, if $\rho : N \rightarrow N \otimes H$, then $\rho(n) = n_{(0)} \otimes n_{(1)}$. Similarly if $N$ is left comodule with $\rho : N \rightarrow H \otimes N$, then  $\rho(n) = n_{(-1)} \otimes n_{(0)}$. 

We will be working with categories of modules and comodules, over algebras and coalgebras respectively. We will use the following general rules for the notation: $$ \mathcal{M}_B \, \, \text{denotes the category of right B-modules} $$ 
                 $$ \mathcal{M}^D \, \, \text{denotes the category of right D-comodules}$$
Similarly: 
$$ _B{\mathcal{M}} \, \, \text{denotes the category of left B-modules}  $$
$$^D{\mathcal{M}} \, \, \text{denotes the category of left D-comodules}.$$
$$$$

\begin{Def}
We say that $A$ is a \textbf{right coideal subalgebra} of $H$ if $A$ is a subalgebra which is also a right coideal of $H$, namely $\Delta_H(A) \subset A \otimes H$. 
\end{Def} Notice that this gives $A$ the structure of a right $H$-comodule.
Given a Hopf algebra $H$ and a coideal subalgebra $A$ we can consider categories carrying both a module and a comodule structure with some compatibility conditions. Following the notation that we used above, we can make the following definition :
\begin{Def}\label{H-comodules, A-modules}
By $\mathcal{M}^H_A$ we will denote the category of right $A$-modules, right $H$-comodules. The objects in this category are vector spaces $M$ such that 
\begin{enumerate}
\item $M$ is a right $A$-module, i.e. there is a module map $a : M \otimes A \rightarrow M$.
\item  $M$ is a right $H$-comodule, i.e there is a comodule map $\rho : M \rightarrow M \otimes H$.
\item The map $a : M \otimes A \rightarrow M$ is a map of $H$-comodules. 

\end{enumerate}
The morphisms are $A$-linear, $H$-colinear maps. 
\end{Def}

\begin{Def}
We say that $C$ is a \textbf{quotient  left $H$-module coalgebra} if $C$ is the quotient of $H$ by a coideal and left ideal $I$.
\end{Def}
Notice that then the induced map $ p : H \otimes C \rightarrow C$ gives $C$ the structure of an $H$-module. Given the projection map $\pi : H \rightarrow H/I = C$, we will usually denote $\pi(h)$ by $\bar{h}$.

As in the case of coideal subalgebras, given a Hopf Algebra $H$ and a quotient left $H$-module coalgebra $C$ we can define the following :
\begin{Def}
 We let $_H^C{\mathcal{M}}$ be the category of left $H$-modules, left $C$-comodules.  The objects in this category are vector spaces $M$ such that 

\begin{enumerate}
\item There is a module map $a : H \otimes M \rightarrow M$
\item There is a comodule map $ \rho : M \rightarrow C \otimes M$
\item The map $\rho : M \rightarrow C \otimes M$ is a map of $H$-modules.

\end{enumerate}
The morphisms are $H$-linear and $C$-colinear maps.

\end{Def}

 Finally we would like to recall the notions of a faithfully flat module and a faithfully coflat comodule. 
\begin{Def} A right module $M$ over an algebra $B$ is \textbf{flat} (respectively \textbf{faithfully flat}) if and only if the functor $M \otimes_B - : \, \,  _B{\mathcal{M}} \rightarrow _k{\mathcal{M}}$ preserves (respectively preserves and reflects) short exact sequences.

\end{Def}

\begin{Def} A right comodule $V$ over a coalgebra $D$ is \textbf{coflat} (respectively \textbf{faithfully coflat}) if and only if the functor$V \underset{D}{\Box} - : \, \, ^D{\mathcal{M}} \rightarrow _k{\mathcal{M}}$ preserves (respectively preserves and reflects) short exact sequences.

\end{Def}Recall that if $V \in \mathcal{M}^D$ and $W \in \,   ^D{\mathcal{M}}$ then the \textbf{cotensor product} $V \underset{D}{\Box} W$ is defined as the kernel of $$\rho_V \otimes id - id \otimes \rho_W : V \otimes W \rightarrow V \otimes D \otimes W.$$


\subsection{General results}

In this section we recall some known facts about coideal subalgebras and quotient coalgebras of Hopf algebras. We will see how the induced module categories defined previously are related when we add the necessary conditions of faithful flatness.

We will first see that any right (or left) coideal in $U$ gives rise to a right (or left) coideal subalgebra of $H$. 

\begin{proposition}[{\cite{Letzter}, Theorem 3.1}] \label{ideal gives coideal subalgebra} Let $U$ be a Hopf algebra with bijective antipode, $\mathcal{C}$ a tensor category of finite-dimensional right $U$-modules and let $H$ be the dual of $U$ with respect to $\mathcal{C}$.  If $Z \subset U$ is a right coideal of $U$, then $A \subset H$ defined by $ A : = \{ h \in H | h \cdot z = \epsilon_U(z) h \, \, \, \text{for all} \, z \in Z \}$ where  $h \cdot z$ is the restriction of the natural action of $U$ on $H$ described above, is a right coideal subalgebra. 

\end{proposition}

Moreover, there is a one-to-one correspondence between right coideal subalgebras of a Hopf algebra $H$ and quotient left module coalgebras $C$.  
\begin{proposition}[{\cite{RelativeHopf}, Proposition 1}] \label{1-1} Let $A$ be a right coideal subalgebra of $H$ and denote $A^+ = A \cap \text{ker}\,\epsilon_H$; then $H_A = H/HA^+$ is a quotient left module coalgebra of $H$. Dually, if $ \pi : H \rightarrow C$ is a quotient left module coalgebra, then $^C{H} = \{ h \in H | \pi(h_{(1)}) \otimes h_{(2)} = \pi(1) \otimes h \} = \, ^{\text{co}C}H  $ is a right coideal subalgebra of $H$. 

\end{proposition}

Combining Propositions \ref{ideal gives coideal subalgebra} and \ref{1-1} we can deduce that whenever we are given a coideal in $U$ we can define a coideal subalgebra of its dual Hopf algebra $H$ or equivalently a quotient left $H$-module coalgebra. Let us see now how the conditions of faithful flatness fit into the picture. 


\begin{theorem} [{\cite{RelativeHopf}, Theorem 2}] Let $H$ be a Hopf algebra and $C$ a quotient left module coalgebra; then $^C{H}$ as defined in Proposition \ref{1-1} is a right coideal subalgebra. Suppose that $H$ is faithfully coflat as a right $C$-comodule. Then 
 $_{^C{H}}{\mathcal{M}} \simeq _H^C{\mathcal{M}}$ and $H$ is faithfully flat as a right $A$-module. 
\end{theorem}

\begin{theorem} [{\cite{Masuoka&Wigner}, Theorem 2.1}] \label{MW} Let $H$ be a Hopf algebra with bijective antipode and $A \subset H$ a right coideal subalgebra. Let $H_A$ be (as in Proposition \ref{1-1})  the corresponding quotient left module coalgebra. Then the following are equivalent:
\begin{enumerate} 
\item $H$ is faithfully coflat as a left $H_A$-comodule.
\item $H$ is faithfully flat as a left $A$-module.
\item There is an equivalence of categories $\mathcal{M}^H _ A \simeq \mathcal{M}^{H_A}$. 
\end{enumerate}
\end{theorem}

From the above we conclude the following:
\begin{corollary} \label{cor1-1} [{\cite{Masuoka}, Theorem 1.11}] :  Let $H$ be a Hopf algebra with bijective antipode. Then $ A \mapsto H_A$, $C \mapsto \,  ^C{H}$ give a one-to-one correspondence between the right coideal subalgebras of $H$ over which $H$ is left faithfully flat and the quotient left $H$-module coalgebras of $H$ over which $H$ is left faithfully coflat. 
\end{corollary}

\subsection{Semisimplicity and cosemisimplicity} \label{ss and cs}

The notion of faithful flatness (respectively faithful coflatness) is an important condition when we want to work with the quantum analogue of an affine homogeneous space. \\ In this direction, we will look at an important result from \cite{Muller&Schneider} that connects some semisimplicity (respectively cosemisimplicity) conditions with the faithful flatness (respectively faithful coflatness) notion that we are looking for. In their work  M\"uller and Schneider use a slightly different approach which we briefly recall here:

Let us start with a Hopf algebra $U$ with bijective antipode and a left coideal subalgebra $K$. Let $\mathcal{C}$ be the tensor category of finite-dimensional left $U$-modules and let $H$ be the dual of $U$ with respect to $\mathcal{C}$. Then $$A := \{ h \in H | h \cdot K^+ =0 \}$$ is a right coideal subalgebra of $H$. Then (as before) $ C = H/HA^+$  is a left module quotient coalgebra and the results from the previous section hold.

\begin{Def} \label{def of C-semisimple} Let $\mathcal{C}$ be a tensor category of left $U$-modules. A subalgebra $K \subset U$ is called $\mathcal{C}$\textbf{-semisimple} if all $V \in \mathcal{C}$ are semisimple as left $K$-modules (by restriction).

\end{Def}

\begin{theorem} [{\cite{Muller&Schneider}, Theorem 2.2}] \label{C-semisimple}  Let $U$ be a Hopf algebra, $K \subset U$ a left coideal subalgebra and $\mathcal{C}$ a tensor category of finite-dimensional left $U$-modules. Let $H$ be the dual Hopf algebra with respect to $\mathcal{C}$, $A := \{ h \in H | h \cdot K^+ =0 \}$ and $C= H/HA^+$. Assume that the antipode of $H$ is bijective. Then:
\begin{enumerate}
\item $A \subset H$ is a right coideal subalgebra with $A = ^{\text{co}C}H$.
\item If $K$ is $\mathcal{C}$-semisimple then $C$ is cosemisimple and $H$ is faithfully flat as a left and right $A$-module.

\end{enumerate}

\end{theorem}

For the full statement of the above theorem we refer to \cite{Muller&Schneider}, Theorem 2.2.

\section{Quantum subgroups -- A definition} \label{def of qs}

\begin{Def} Let $H$ be a Hopf algebra with bijective antipode. A \textbf{quantum homogeneous space} is a right coideal subalgebra $A$ of $H$ over which $H$ is faithfully flat. 
\end{Def}

This definition agrees with the classical picture. Indeed, it is known that a flat morphism of commutative rings $f: R \rightarrow S$ is faithfully flat if and only if the dual map $f* : \text{Spec}S \rightarrow \text{Spec}R$ is surjective (see for example \cite{matsumura} Theorem 7.3). In the same spirit is Lemma 2.2 of \cite{Richardson}: 

\begin{lemma}  Let $\phi: X \rightarrow Y$ be a surjective morphism of smooth irreducible affine algebraic varieties and let $s=\text{dim}X-\text{dim}Y$.  Assume that for every $y \in Y$, each irreducible component of the fibre $\phi^{-1}(y)$ is $s$-dimensional. Then $\mathcal{O}(X)$ is a fuithfully flat $\mathcal{O}(Y)$-module.
\end{lemma}

Keeping in mind the above definition and the previous results we can proceed to the following construction: Let $I$ be a right coideal of $U$. By Proposition \ref{ideal gives coideal subalgebra} this gives rise to a right coideal subalgebra $A$ of the dual Hopf algebra $H$. By Proposition \ref{1-1} this gives us a quotient left module coalgebra $H_A$ of $H$. If moreover $H$ is faithfully flat over $A$ (that is, if $A$ corresponds to a quantum homogeneous space), by Theorem \ref{MW} we conclude that there is an equivalence of categories $\mathcal{M}^H _ A \simeq \mathcal{M}^{H_A}$. Therefore if the coideal $I$ with which we started is associated to a classical subgroup $M$ of a group $G$, then the quotient coalgebra $H_A$ can be used as the definition of the quantum subgroup corresponding to $M$. In this setting $U$ is the quantized universal enveloping algebra of the Lie algebra of $G$ and $H$ is the quantum coordinate algebra $\mathcal{O}_q(G)$. 

This leads to the following definition:

\begin{Def} \label{quantum reductive subgroup} Let $H$ be a Hopf algebra with bijective antipode. A \textbf{quantum subgroup} is a quotient left $H$-module coalgebra $C$ of $H$ such that $H$ is faithfully coflat over $C$. 
\end{Def}

We would like to point out here, complementing the remarks about the homogeneous spaces above, that with the above definition we aim to quantize subgroups $H$ of a group $G$ such that $G/H$ is affine. As we saw, classically, the faithfully flatness condition corresponds exactly to these subgroups for which the quotient space is affine. 
It is also worth mentioning that if we restrict to the case where the ambient group is reductive then by Matsushima's criterion we know that $G/H$ is affine if and only if $H$ is reductive. So in this case the two notions coincide. 

In this latter case the following can be useful:

\begin{remark} If $H$ is cosemisimple, which is the case when we work with the quantum coordinate algebra $\mathcal{O}_q(G)$ of a reductive group for generic $q$, and if $A$ is not only a coideal subalgebra, but a Hopf subalgebra of $H$, then by the main result in \cite{Chivar} $H$ is faithfully flat over $A$. 
\end{remark}

\section{Module categories}

In this section we recall the definitions and establish the notation for monoidal and module categories. Our main references are \cite{Bakalov+Kirillov}, \cite{Kirillov+Ostrik} and \cite{Ostrik}.
We will assume throughout that all categories are abelian over an algebraically closed field $k$. All functors are assumed to be additive. 

\subsection{Monoidal categories} We start with the definition of a monoidal category. 
\begin{Def}\label{monoidal}  A \textbf{monoidal category} consists of the following:
\begin{enumerate}
\item A category $\mathcal{C}$,
\item a functor $\otimes : \mathcal{C} \times \mathcal{C} \rightarrow \mathcal{C}$,
\item functorial isomorphisms $a_{X,Y,Z} : ( X \otimes Y) \otimes Z \rightarrow X \otimes (Y \otimes Z)$,
\item a unit object $\mathbf{1} \in \mathcal{C}$, and
\item functorial isomorphisms $r_X : X \otimes \mathbf{1} \rightarrow X$ and $l_X : \mathbf{1} \otimes X \rightarrow X$,
\end{enumerate}
such that the following diagrams: 
 
$$
\text{
\xymatrix{ &{((X\otimes Y) \otimes Z) \otimes W} \ar[ld]_{a_{X,Y,Z} \otimes id}  \ar[rd]^{a_{X \otimes Y, Z, W}} &{} \\
{(X \otimes (Y \otimes Z)) \otimes W} \ar[d]_{a_{X,Y \otimes Z , W}} &{}  &{(X \otimes Y) \otimes (Z \otimes W)} \ar[d]^{a_{X,Y,Z \otimes W}} \\
{X \otimes ((Y \otimes Z) \otimes W)} \ar[rr]_{id \otimes a_{Y,Z,W}} &{} & {X \otimes (Y \otimes ( Z \otimes W))} } } $$

and 

$$
\text{
\xymatrix{ {(X \otimes \mathbf{1}) \otimes Y} \ar[rr]^{a_{X, \mathbf{1}, Y}} \ar[rd]_{r_X \otimes id}&{} & {X \otimes (\mathbf{1} \otimes Y)} \ar[ld]^{id \otimes l_Y} \\
 &{ X \otimes Y} &{} }}$$

commute.
\end{Def} 

\begin{Def}
Let $\mathcal{C}_1, \mathcal{C}_2$ be two monoidal categories. A \textbf{monoidal functor} between $\mathcal{C}_1$ and $\mathcal{C}_2$ is a triple $(F, b,u)$ consisting of a functor $F : \mathcal{C}_1 \rightarrow \mathcal{C}_2$, functorial isomorphisms $b_{X,Y} : F(X \otimes Y) \rightarrow F(X) \otimes F(Y)$ and isomorphism $u : F(\mathbf{1}) \rightarrow \mathbf{1}$ such that the following diagrams: 
$$$$
$$\begin{CD}
{F((X \otimes Y) \otimes Z)} @>{b_{X \otimes Y, Z}}>>     {F(X \otimes Y) \otimes F(Z)}     @> {b_{X,Z} \otimes id}>> {(F(X) \otimes F(Y)) \otimes F(Z)}\\
@VV {F_{a_{X,Y,Z}} }V @. @VV{a_{F(X), F(Y),F(Z)}}V   \\
{F(X \otimes (Y \otimes Z))}      @>{b_{X,Y \otimes Z}}>> {F(X) \otimes F(Y \otimes Z)}   @>{id \otimes b_{Y,Z}}>> {F(X) \otimes (F(Y) \otimes F(Z))} 
\end{CD} $$
\\\\
and

$$
\begin{CD}
{F(\mathbf{1} \otimes X)} @>{b_{\mathbf{1} \otimes X}}>> {F(\mathbf{1}) \otimes F(X)} @. {} @. {F(X \otimes \mathbf{1})} @>{b_{X,\mathbf{1}}}>> {F(X) \otimes F(\mathbf{1})} \\
@VV{F(l_X)}V @VV{u \otimes id}V @. @VV{F(r_X)}V @VV{id \otimes u}V \\
{F(X)} @>{l_{F(X)}}>> {\mathbf{1} \otimes F(X)} @. {\,\,\,\,\,\,\,\,\,\,\,\,\,\,\,\,\,\,\,\,\,\,\,\,} @. {F(X)} @>{r_{F(X)}}>> {F(X) \otimes \mathbf{1}}
\end{CD} $$

commute.

\end{Def}

Notice that in some other texts, a functor of the form defined above, for which $b$ and $u$ are required to be isomorphisms, is often called a strict monoidal functor.

\begin{example} If $H$ is a Hopf algebra, then the category of $H$-comodules is a monoidal category. Indeed, given two $H$-comodules $V$ and $W$, we can consider $V \otimes W$ in $\mathcal{M}^H$ with comodule structure given as follows:
$$
\text{ \xymatrixcolsep{5pc} \xymatrix { V \otimes W \ar[r]^{\rho_V \otimes \rho_W} & {V \otimes H \otimes W \otimes H} \ar[r]^{id \otimes \text{flip} \otimes id} &{V \otimes W \otimes H \otimes H} \ar[r]^{\,\,\,\,\,\,\,\,\,\,\,\,\,\, \,\,\, id \otimes id \otimes m_H} & {} \\{} & {} \ar[r]^{id \otimes id \otimes m_H} & {V \otimes W \otimes H} &{} }}$$
\end{example}

We finish this section by stating a theorem that will be useful later in order to simplify arguments in some proofs. First, we need a definition.

\begin{Def}
A monoidal category $\mathcal{C}$ is called \textbf{strict} if for all objects $X,Y, Z \in \mathcal{C}$ the functorial isomorphisms (from Definition \ref{monoidal}) $a_{X,Y,Z} , r_X$ and $l_X$ are the identity isomorphisms. In this case we have $ X \otimes \mathbf{1} = X$, $ \mathbf{1} \otimes X = X$ and $ ( X \otimes Y) \otimes Z = X \otimes (Y \otimes Z)$ , and similarly for multiple tensor products. 
\end{Def}

\begin{theorem} \label{monoidal strict}
Every monoidal category is equivalent to a strict one.
\end{theorem}
A proof of this can be found in \cite{MacLane}.

\subsection{Module categories}
\begin{Def} A \textbf{module category} over a monoidal category $\mathcal{C}$ consists of the following:
\begin{enumerate}
\item A category $\mathcal{M}$,
\item an exact bifunctor $\otimes : \mathcal{C} \times \mathcal{M} \rightarrow \mathcal{M}$,
\item functorial isomorphisms $m_{X,Y,M} : (X \otimes Y) \otimes M \rightarrow X \otimes (Y \otimes M)$, for $X, Y \in \mathcal{C}$ and $M \in \mathcal{M}$ and
\item functorial isomorphisms $l_M : \mathbf{1} \otimes M \rightarrow M, $

\end{enumerate}
such that the following diagrams: 

$$
\text{
\xymatrix{ &{((X\otimes Y) \otimes Z) \otimes M} \ar[ld]_{a_{X,Y,Z} \otimes id}  \ar[rd]^{m_{X \otimes Y, Z, M}} &{} \\
{(X \otimes (Y \otimes Z)) \otimes M} \ar[d]_{m_{X,Y \otimes Z , M}} &{}  &{(X \otimes Y) \otimes (Z \otimes M)} \ar[d]^{m_{X,Y,Z \otimes M}} \\
{X \otimes ((Y \otimes Z) \otimes M)} \ar[rr]_{id \otimes m_{Y,Z,M}} &{} & {X \otimes (Y \otimes ( Z \otimes M))} } } $$

and 

$$
\text{
\xymatrix{ {(X \otimes \mathbf{1}) \otimes M} \ar[rr]^{m_{X, \mathbf{1}, M}} \ar[rd]_{r_X \otimes id}&{} & {X \otimes (\mathbf{1} \otimes M)} \ar[ld]^{id \otimes l_M} \\
 &{ X \otimes M} &{} }}$$

commute.
\end{Def}

\begin{Def} Let $\mathcal{M}_1, \mathcal{M}_2$ be two module categories over a monoidal category $\mathcal{C}$. A \textbf{module functor} from $\mathcal{M}_1$ to $\mathcal{M}_2$ is a functor $F : \mathcal{M}_1 \rightarrow \mathcal{M}_2$ together with functorial isomorphisms $c_{X,M} : F(X \otimes M) \rightarrow X \otimes F(M)$ for every $X \in \mathcal{C}$ and $ M \in \mathcal{M}_1$,  such that the following diagrams:

$$
\text{
\xymatrix{ &{F(( X \otimes Y) \otimes M)} \ar[ld]_{F{m_{X,Y,M}}}  \ar[rd]^{c_{X \otimes Y, M}} &{} \\
{F(X \otimes (Y \otimes M))} \ar[d]_{c_{X,Y \otimes M}} &{}  &{(X \otimes Y) \otimes F(M)} \ar[d]^{m_{X,Y,F(M)}} \\
{X \otimes F( Y \otimes M)} \ar[rr]_{id \otimes c_{Y,M}} &{} & {X \otimes (Y \otimes F(M))} } } $$

and 

$$
\text{
\xymatrix{ { F( \mathbf{1} \otimes M)} \ar[rr]^{Fl_M} \ar[rd]_{c_{\mathbf{1}, M}}&{} & {F(M)} \\
 &{ \mathbf{1} \otimes F(M)} \ar[ru]^{l_{F(M)}} &{} }}$$

commute. 
\end{Def}
 Notice that what we have defined is often called a strict module functor in other texts, because we are requiring that the morphisms $c_{X,M}$ are isomorphisms. 

\begin{example} Any monoidal category $\mathcal{C}$ can be viewed as a module category over itself, with associativity and unit isomorphisms given by the ones from the monoidal structure. 

\end{example}

\subsection{Algebras in monoidal categories}
\begin{Def} \label{algebra in monoidal}
An \textbf{algebra} in a monoidal category $\mathcal{C}$ is an object $A \in \mathcal{C}$ together with a multiplication morphism $m : A \otimes A \rightarrow A$ and a unit morphism $ e : \mathbf{1} \rightarrow A$ such that the following diagrams: 

$$
\text{
\xymatrix{ {} & {(A \otimes A) \otimes A} \ar[ld]_{a_{A,A,A}}  \ar[rd]^{m \otimes id} &{} \\
{A \otimes (A \otimes A)} \ar[d]_{id \otimes m} &{}  &{ A \otimes A} \ar[d]^{m} \\
{ A \otimes A} \ar[rr]_{m} &{} & {A} } } $$

and 

$$
\text{
\xymatrix{ {} & {\mathbf{1} \otimes A} \ar[ld]_{l_A} \ar[rd]^{e \otimes id} & {} &{} &{ A \otimes \mathbf{1}} \ar[ld]_{r_a} \ar[rd]^{id \otimes e} &{} \\
A & {} & { A \otimes A} \ar[ll]^{m} &{A} &{} &{A \otimes A} \ar[ll]^{m}}}$$

commute.
\end{Def}

\begin{example} If $\mathcal{C}$ is the monoidal category of $H$-comodules for a Hopf algebra $H$, and if $A$ is a coideal subalgebra of $H$, then $A$ is an algebra for the monoidal category $\mathcal{C}$. 

\end{example}

\begin{Def} \label{modules for algebra in monoidal} A \textbf{right module} over an algebra $A$ in a monoidal category $\mathcal{C}$ is an object $M \in \mathcal{C}$ together with an action morphism $ a : M \otimes A \rightarrow M$ such that the following diagrams:

$$
\text{
\xymatrix{ {M \otimes A \otimes A} \ar[r]^{id \otimes m} \ar[d]^{a \otimes id} &{M \otimes A} \ar[d]_{a} &{} &{} &{M \otimes \mathbf{1}} \ar[ld]_{r_M} \ar[rd]^{id \otimes e} &{} \\
{M \otimes A} \ar[r]^{a} &{M} &{} &{M} &{} &{M \otimes A} \ar[ll]^{a} }}$$

commute.
Similarly we can define the notion of a left module.

\end{Def}

Finally,
\begin{Def} A \textbf{morphism} between two right modules $M_1, M_2$ over $A$ is a morphism in $\mathcal{C}$ such that the diagram commutes:
$$
\begin{CD}
{M_1 \otimes A} @>{a \otimes id}>> {M_2 \otimes A}\\
@VV{a_1}V @VV{a_2}V\\
{M_1} @>a>> {M_2}
\end{CD}
$$
There is a similar definition for left modules.
\end{Def}

If $A$ is an algebra in a monoidal category $\mathcal{C}$, we will denote the category of right $A$-modules by $\text{Mod}_{\mathcal{C}}(A)$. This is an abelian category. 

\begin{remark}\label{A-modules are a module category} Using the above notations, $ \text{Mod}_{\mathcal{C}}(A)$ can be endowed with the structure of a module category over $\mathcal{C}$.
\end{remark}
Indeed, let $M$ be a right $A$-module and let $X \in \mathcal{C}$. We want to define a functor $\otimes : \mathcal{C} \times \text{Mod}_{\mathcal{C}}(A) \rightarrow \text{Mod}_{\mathcal{C}}(A)$. Since $X, M \in \mathcal{C}$, we can consider the object $ X \otimes M \in \mathcal{C}$ (using the structure of $\mathcal{C}$). We see that $ X \otimes M$ is also a right $A$-module with action morphism given by $id \otimes a_M$. All necessary properties of the module category $\text{Mod}_{\mathcal{C}}(A)$ now follow from the monoidal structure of $\mathcal{C}$.

\begin{example} \label{They define same category} Suppose that $\mathcal{C}$ is the monoidal category of right $H$-comodules for a Hopf algebra $H$,  and that $A$ is a coideal subalgebra. Then $\text{Mod}_{\mathcal{C}}(A)$ is equal to the category $\mathcal{M}^H_A$ defined in \ref{H-comodules, A-modules}.

\end{example}

Using Remark \ref{A-modules are a module category} and Example \ref{They define same category} above we can state the following

\begin{corollary}\label{coideal subalgebra gives module category}
Any coideal subalgebra $A$  of a Hopf algebra $H$ gives rise to a module category over the category of $H$-comodules.
\end{corollary}


\section{Module categories of the form $\text{Mod}_{\mathcal{C}}(A)$} \label{prwto}

The previous section finished with the interesting Corollary \ref{coideal subalgebra gives module category}. It is natural to ask under what conditions we could have a converse version of it. Specifically, we need to find answers to two questions. Firstly, given a monoidal category $\mathcal{C}$, and a module category $\mathcal{M}$ over it, when is $\mathcal{M}$ equivalent to $\text{Mod}_{\mathcal{C}}(A)$  for an algebra $A$ in $\mathcal{C}$? 
And secondly, in the specific case where $\mathcal{C}$ is considered to be the category of $H$-comodules for a Hopf algebra $H$, when does an algebra $A$ in $\mathcal{C}$ correspond to a coideal subalgebra of $H$?
In the following sections, we will try to find an answer to these two questions.

\subsection{Background material}
We start by recalling the theory of monads and comonads.

\begin{Def} \label{Monad} A \textbf{monad} $T$ on a category $\mathcal{C}$ is an endofunctor $T : \mathcal{C} \rightarrow \mathcal{C}$ together with two natural transformations $\mu : T^2 \rightarrow T$ and $\eta : id_{\mathcal{C}} \rightarrow T$ such that the following diagrams:
$$
\begin{CD}
{T^3} @>{T\mu}>> {T^2}\\
@VV{\mu T}V @VV{\mu}V\\
{T^2} @>\mu>> {T}
\end{CD}
$$

and 

$$
\text{
 \xymatrix{
T  \ar[rd]_{=} \ar[r]^{T \eta} &  T^2 \ar[d]_{ \mu} & T \ar[l]_{\eta T} \ar[ld]^{=}\\
&{T} }}$$ 
 
commute.

\end{Def}

\begin{remark} \label{notation in diagram}
In the above definition, by $T\mu$ at an object $X$ we mean $T(\mu_X)$, and by $\mu T$ at an object $X$ we mean $\mu_{T(X)}$, where in general $\mu_Y$ denotes the morphism $\mu : T^2(Y) \rightarrow T(Y)$. The same notation is being used for $\eta$. 
\end{remark}

Monads are closely related to adjoint functors. Specifically, it is known that a pair of adjoint functors gives rise to a monad.
\begin{theorem}
Let $G : \mathcal{C} \rightarrow \mathcal{D}$ be a functor between two categories which admits a right adjoint $U : \mathcal{D} \rightarrow \mathcal{C}$ with adjunction morphisms $\eta : id \rightarrow UG$ and $\epsilon : GU \rightarrow id$. Then $T = (UG, \eta, U\epsilon G)$ is a monad on $\mathcal{C}$.  

\end{theorem}

For a proof of this very well known fact, see for example \cite{toposestriples}.

\begin{Def} Let $(T, \mu, \eta)$ be a monad on a category $\mathcal{C}$. A \textbf{$T$-algebra} is a pair $(N,\lambda)$ where $N \in \mathcal{C}$ and $\lambda : TN \rightarrow N$ is a morphism in $\mathcal{C}$ such that the following two diagrams: 
$$
\text{ \xymatrix { N \ar[rd]_{id_{N}} \ar[r]^{\eta N} & {TN} \ar[d]^{\lambda}    &{}  & {T^2N} \ar[d]_{\mu N} \ar[r]^{T\lambda} & {TN} \ar[d]^{\lambda}  \\
& N &{} &{TN} \ar[r]_{\lambda} & N   }}$$
commute.
A map $f : (N, \lambda) \rightarrow (B, {\lambda}')$ is a map $f : N \rightarrow B$ in $\mathcal{C}$ such that $f \circ \lambda = {\lambda}' \circ Tf$.
\end{Def}

We denote the category of all $T$-algebras for a monad $T$ in a category $\mathcal{C}$ by $\mathcal{C}^T$.

\begin{remark} \label{object as T-algebra} We remark that every object $X \in \mathcal{C}$ gives rise to a $T$-algebra $(T(X), \mu_X)$.

\end{remark}

The assignment $ X \mapsto (T(X), \mu_X)$ yields a functor $F^T : \mathcal{C} \rightarrow \mathcal{C}^T$ with $F^Tf = Tf$. 
On the other hand every $T$-algebra $(N, \lambda)$ can be seen as an object in $\mathcal{C}$ by considering the underlying object $N$ and forgetting the structure given by $\lambda$. This gives a functor $U^T : \mathcal{C}^T \rightarrow \mathcal{C}$ with $U^Tf =f$. 

$(F^T, U^T, \eta, \epsilon_T)$ is a pair of adjoint functors and it is easy to see that it defines the given monad $(T, \mu, \eta)$.
Moreover it satisfies a universal property: \\ \\ For every adjoint pair $ (G, U, \eta, \epsilon)$ between two categories $\mathcal{C}$ and $\mathcal{D}$ that defines the monad $(T, \mu, \eta)$ there is a unique functor $K : \mathcal{D} \rightarrow \mathcal{C}^T$, called the \textbf{comparison functor}, such that $KG=F^T$ and $U^T K = U$:
$$
\text{ \xymatrixcolsep{5pc}
\xymatrix{ {\mathcal{C}} \ar@/^/[r]^{F^T} \ar@/^/[d]^{G} & {\mathcal{C}^T} \ar@/^/[l]_{U^T} \\
{\mathcal{D}} \ar@/^/[u]^{U} \ar[ur]_{K} &{} }}$$
 (See also \cite{MacLane}). \\ \\

\begin{Def} We say that an adjoint pair $(G, U, \eta, \epsilon)$ is \textbf{monadic} if the functor $K : \mathcal{D} \rightarrow \mathcal{C}^T$ is an equivalence. 
\end{Def}

The very well known Barr-Beck Monadicity theorem gives conditions under which an adjoint pair is monadic. Here, we give a version of the Barr-Beck theorem for abelian categories which can be found in \cite{beckforabelian}.
\begin{theorem} \label{Barr-Beck}
Let $G: \mathcal{C} \rightarrow \mathcal{D} $ be an additive functor which admits a right adjoint $U : \mathcal{D} \rightarrow \mathcal{C}$ which is exact. Denote the defined monad by $T$. Then the comparison functor $K : \mathcal{D} \rightarrow \mathcal{C}^T$ is an equivalence if and only if $U$ is faithful.

\end{theorem}

By dualizing the above definitions and constructions we can define comonads and coalgebras for a comonad. 

The Barr-Beck theorem for comonads has the following form :
\begin{theorem} \label{Barr-Beck for comonads}
Let $F: \mathcal{C} \rightarrow \mathcal{B} $ be an additive functor which is exact and faithful. Assume also that it admits a right adjoint $U : \mathcal{B} \rightarrow \mathcal{C}$. Denote the defined comonad on $\mathcal{B}$ by $G$. Then $\mathcal{C}$ is equivalent to the category $\mathcal{M}^G$ of $G$-coalgebras. 

\end{theorem}

\subsection{Main theorem}

\begin{theorem} \label{theorem 1}Let $ (\mathcal{C}, \otimes, I)$ be a monoidal category and let $\mathcal{M}$ be a module category over $\mathcal{C}$. Assume that there is a pair of adjoint functors $\text{Res} : \mathcal{C} \rightarrow \mathcal{M}$ and $\text{Ind} : \mathcal{M} \rightarrow \mathcal{C}$ such that both Res and Ind are morphisms of module categories. Assume further that Ind is exact and faithful. Then $\mathcal{M}$ is equivalent to $\text{Mod}_{\mathcal{C}}(A)$ for an algebra $A \in \mathcal{C}$. 

\end{theorem}

\begin{proof}
Without loss of generality (see also Theorem \ref{monoidal strict}), we can assume that $\mathcal{C}$ is a strict monoidal category. 

We start by noticing that $Res$ and $Ind$ define a monad $T = Ind \circ Res$ in $\mathcal{C}$.

We will split the proof in steps. 
\begin{description}
\item[Step 1.]   \emph{We show that $T(V)= V \otimes T(I)$ for all $V \in \mathcal{C}$.}\\ Since both $Ind$ and $Res$ are functors of module categories we have the following: $$ T ( M \otimes N) \simeq Ind ( Res (M \otimes N)) \simeq Ind (M \otimes Res(N)) \simeq M \otimes Ind(Res(N)) = $$ $= M \otimes T(N)$,  for every  $M , N \in \mathcal{C}$. 

In particular, for every $V \in \mathcal{C}$ we have $ V = V \otimes I$, hence $T(V) = T(V \otimes I) \simeq V \otimes T(I)$. Let us define $ G: \mathcal{C} \rightarrow \mathcal{C}$ by $G = id \otimes T(I)$. Then by what we did above we see that $T$ is naturally equivalent to $G$ but to make notation simpler we will assume that $T=G$. We will use this notation later.

\item[Step 2.] \emph{We show that $T(I)$ is an algebra in $\mathcal{C}$.} \\ Since $T$ is a monad, there are natural transformations $ \mu : T ^2 \rightarrow T$ and $\eta : id \rightarrow T$ which, as we saw before (Remark \ref{object as T-algebra}), endow every object $T(V)$ with the structure of a $T$-algebra. In particular, the pair $(T(I), \mu_I)$ is a $T$-algebra. This means that there exists a morphism $ {\mu}_I : T^2(I) \rightarrow T(I)$. By Step 1, $T^2(I) \simeq T(I) \otimes T(I)$ hence $\mu_I : T(I) \otimes T(I) \rightarrow T(I)$. Moreover, by $\eta$ we get a morphism ${\eta}_I : I \rightarrow T(I)$.
It remains to show that $(T(I), \mu_I, \eta_I)$ satisfies the commutative diagrams from Definition \ref{algebra in monoidal}. 
For the associativity diagram, we use the first diagram of Definition \ref{Monad} of a monad at the object $I$. This gives us the following:
$$
\begin{CD}
{T(I) \otimes T(I) \otimes T(I)} @>{T \mu_I}>> {T(I) \otimes T(I)}\\
@VV{\mu_{T(I)} }V @VV{\mu_I}V\\
{T(I) \otimes T(I)} @>\mu_I>> {T(I)}
\end{CD}
$$

Recall now the functor $G$ that we defined at at the end of Step 1. The natural transformations $\mu$ and $\eta$ induce natural transformations $\tilde{\mu} : G^2 \rightarrow G$ and  $\tilde{\eta} : id \rightarrow G$, so that $\tilde{\mu} = id \otimes \mu_I$ and $\tilde{\eta} = id \otimes \eta_I$. But $G=T$ and therefore $\mu = id \otimes \mu_I$ and similarly $\eta = id \otimes \eta_I$. 

In particular $\mu_{T(I)} = id \otimes \mu_I$. Furthermore, $T\mu_I = \mu_I \otimes id$. So the diagram above gives us exactly what we wanted. 

Similarly, the commutativity of the diagrams for $\eta_I$ follows directly from the second diagram in Definition \ref{Monad}, which at the object $T(I)$ gives:

$$
\text{ \xymatrixcolsep{5pc}
 \xymatrix{
{T(I)}  \ar[rd]_{=} \ar[r]^{T \eta_I} &  {T(I) \otimes T(I)} \ar[d]_{ \mu_I} & {T(I)} \ar[l]_{\eta_I T} \ar[ld]^{=}\\
&{T(I)} }}$$ 

But $\mathcal{C}$ is strict ($T(I) = T(I) \otimes I = I \otimes T(I)$ and $id = r_{T(I)} = l_{T(I)}$). Moreover $ T(\eta_I) = \eta_I \otimes id$ and, as we noticed above, $\eta_{T(I)} = id \otimes \eta_I$, which gives us :

$$
\text{ \xymatrixcolsep{5pc}
 \xymatrix{
{I \otimes T(I)}  \ar[rd]_{l_{T(I)}} \ar[r]^{\eta_I \otimes id} &  {T(I) \otimes T(I)} \ar[d]_{ \mu_I} & {T(I) \otimes I} \ar[l]_{id \otimes \eta_I} \ar[ld]^{r_{T(A)}}\\
&{T(I)} }}$$ 

Again, this is exactly what we wanted.

\begin{itemize}
\item From now on we will denote $T(I)$ by $\mathbf A$. 
 \end{itemize}

\item[Step 3.] \emph{We show that $\mathcal{M}$ is equivalent to the category of $T$-algebras}. \\ Since $Ind$ is exact and faithful, by the Barr-Beck theorem for abelian categories, Theorem \ref{Barr-Beck},  we can conclude that $\mathcal{M}$ is equivalent to the category of $T$-algebras. 

\item[Step 4.] \emph{We show that the category of $T$-algebras is equivalent to $\text{Mod}_{\mathcal{C}}(A)$.}\\
Recall the diagrams from the definition of a $T$-algebra. For $T(N) = N \otimes A$ they give:

$$
\text{ \xymatrix { N \ar[rd]_{id_{N}} \ar[r]^{\eta N} & {N \otimes A} \ar[d]^{\lambda}    &{}  & {N \otimes A \otimes A} \ar[d]_{\mu N} \ar[r]^{T\lambda} & {N \otimes A} \ar[d]^{\lambda}  \\
& N &{} &{N \otimes A} \ar[r]_{\lambda} & N   }}$$

Using the identities from Step 2 (and recalling that $N = N \otimes I$ and $r_N = id$) we can rewrite this as:

$$
\text{ \xymatrix { {N \otimes I} \ar[rd]_{r_N} \ar[r]^{id \otimes \eta_I} & {N \otimes A} \ar[d]^{\lambda}    &{}  & {N \otimes A \otimes A} \ar[d]_{id \otimes \mu_I} \ar[r]^{\lambda \otimes id} & {N \otimes A} \ar[d]^{\lambda}  \\
& N &{} &{N \otimes A} \ar[r]_{\lambda} & N   }}$$

Similarly, morphisms between $T$-algebras are maps $f : N \rightarrow B$ in $\mathcal{C}$ such that $f \circ \lambda = {\lambda}' \circ Tf$. Since $T(N) = N \otimes A$ and $Tf=f \otimes id$  this gives us: 

$$
\begin{CD}
{N \otimes A} @>{f \otimes id}>> {B \otimes A}\\
@VV{\lambda}V @VV{{\lambda}'}V\\
{N} @>a>> {B}
\end{CD}
$$

But these are exactly the definitions of right $A$-modules and their morphsims. Therefore we have indeed shown that the category of $T$-algebras is the same as the category $\text{Mod}_{\mathcal{C}}(A)$ of $A$-modules for the algebra $A$ in $\mathcal{C}$. \\ \\  This completes our proof.

\end{description}

\end{proof}

\begin{remark} We notice that the converse of Theorem \ref{theorem 1} is also true. Indeed, let $\mathcal{C}$ be a monoidal category and $\text{Mod}_{\mathcal{C}}(A)$ a module category of $\mathcal{C}$ with the action defined in \ref{A-modules are a module category}. We can view $\mathcal{C}$ as a module category over itself using its monoidal structure. Then the functors $\text{Res} : \mathcal{C} \rightarrow \text{Mod}_{\mathcal{C}}(A)$ with $\text{Res}(X) = X \otimes A$ and $\text{Ind} : \text{Mod}_{\mathcal{C}}(A) \rightarrow \mathcal{C}$ with $\text{Ind}$ being the forgetful functor (forgetting the $A$-module structure of objects) are clearly functors of module categories. Moreover $\text{Ind}$ is then exact and faithful.

\end{remark}

\subsection{Properties of $\text{Mod}_{\mathcal{C}}(A)$ when $\mathcal{C}$ is $\mathcal{M}^H$ for a Hopf algebra $H$ }

We will investigate some properties of the module categories $\text{Mod}_{\mathcal{C}}(A)$ when $\mathcal{C}$ is the category of $H$-comodules for a Hopf algebra $H$.

\begin{theorem} \label{category equivalent to C comodules}
Let $\mathcal{D}$ be an abelian $k$-category which is cocomplete and locally presentable. Let $F$ be an additive, cocontinuous, exact and faithful functor to the category Vect of vector spaces. Suppose moreover that the right adjoint of $F$ is cocontinuous. Then $\mathcal{D}$ is equivalent to the category of $C$-comodules for a $k$-coalgebra $C$.

\end{theorem}

\begin{proof}

First we notice that $F$ has a right adjoint by the adjoint functor theorem for locally presentable categories. 
Let us denote this adjoint by $Q$ and assume that $Q$ is cocontinuous.

We let $G = F \circ Q$ be the corresponding comonad on Vect. $G$ is cocontinuous as a composition of cocontinuous functors. It is moreover an endofunctor on Vect. We will show that $G$ is isomorphic to tensoring with $G(\C)$. Indeed, let $V$ be a vector space. $V$ is then the colimit of finite-dimensional vector spaces $V_i$ and each such $V_i$ can be assumed to be isomorphic to ${\C}^{n_i}$ for some $n_i \in \N$. Therefore $G(V) = G(\varinjlim {\C}^{n_i}) = \varinjlim G({\C}^{n_i})$ since $G$ is cocontinuous. But $G({\C}^{n_i}) = G(\C)^{n_i}$ since $G$ is additive. If we let $C = G(\C)$, we get that $G(\C)^{n_i} = C^{n_i} = \C^{n_i} \otimes C$. Therefore,  $\varinjlim G({\C}^{n_i}) = \varinjlim ( \C^{n_i} \otimes C) = \varinjlim {\C}^{n_i} \otimes C = V \otimes C$.

Now, as in the proof of Theorem \ref{theorem 1}, we can show that the comonad $G$ endows $C$ with the structure of a $G$-coalgebra, which in particular makes $C$ a $k$-coalgebra. Also, it is then easy to see (following the same arguments) that the category of $G$-coalgebras is equivalent to the category of $C$-comodules. By the Barr-Beck theorem for comonads, since $F$ is exact and faithful we can deduce that $\mathcal{D}$ is equivalent to $G$-coalgebras, and therefore to $\mathcal{M}^C$ for the coalgebra $C$.

\end{proof}

\begin{lemma} \label{locally presentable} Let $\mathcal{C}$ be the category of $H$-comodules for a Hopf algebra $H$.  Then $\text{Mod}_{\mathcal{C}}(A)$, for an algebra $A \in \mathcal{C}$,  is a locally presentable and cocomplete category.

\end{lemma}

\begin{proof}

$\text{Mod}_{\mathcal{C}}(A)$  is obviously cocomplete.  
So we only need to prove that it is locally presentable.
Let $M \in \text{Mod}_{\mathcal{C}}(A)$. This is an $H$-comodule endowed with an action map $a: M \otimes A \rightarrow M$ which is also a map of comodules. We take $M_i$ a finite-dimensional $H$-subcomodule and consider the $A$-submodule it generates. We denote this by $<M_i>$. Notice that $M_i \subseteq <M_i>$ since $A$ has a unit.  The restriction of the action map $a$ to $M_i \otimes A$ gives a surjection to $<M_i>$ and therefore $<M_i>$ is also an $H$-subcomodule as it is the image of a subcomodule. Now we consider a directed set $I$ and to each $i \in I$ we associate $<M_i>$; whenever $i \leqslant  j$ we consider the inclusion $ \theta_i^j : <M_i> \rightarrow <M_j>$. Since $M = \varinjlim M_i$ as $H$-comodules it follows that $M = \varinjlim <M_i> $. Now we turn our attention to the $<M_i>$s in  $\text{Mod}_{\mathcal{C}}(A)$.  First we notice that any object of the form $M_i \otimes A$ for a finite-dimensional $H$-subcomodule $M_i$ is compact in $\text{Mod}_{\mathcal{C}}(A)$. Indeed, by the adjunction $(Res, Ind)$ between $\text{Mod}_{\mathcal{C}}(A)$  and $\mathcal{M}^H$ we have that $$\text{Hom}_{\text{Mod}_{\mathcal{C}}(A)} (M_i \otimes A, N) \simeq \text{Hom}_{\mathcal{M}^H}(M_i, N).$$ Moreover the finite dimensional $H$-comodules are compact objects in $\mathcal{M}^H$.  Indeed, let $V$ be a finite-dimensional $H$-comodule and consider $\text{Hom}_{\mathcal{M}^H}(V, \varinjlim N_j)$. Let $f$ be in $\text{Hom}_{\mathcal{M}^H}(V, \varinjlim N_j)$. We will show that $f$ factors through one of the inclusions $N_j \rightarrow \varinjlim N_j $. Since $f$ is an $H$-comodule map it is also $k$-linear. Therefore $f \in  \text{Hom}_{\text{Vect}}(V, \varinjlim N_j)$. But now $V$ is a finite-dimensional vector space and hence compact, which means that $f$ factors through some $N_j$. Since $N_j$ is an $H$-comodule by assumption we conclude that $V$ is compact in $\mathcal{M}^H$ as well. So in particular the finite-dimensional $H$-subcomodule $M_i$ is compact and therefore $\text{Hom}_{\mathcal{M}^H}(M_i, \varinjlim N_d) \simeq \varinjlim \text{Hom}_{\mathcal{M}^H}(M_i, N_d)$ for every colimit $\varinjlim N_d$ over a filtered category $D$. So the same holds for the objects $M_i \otimes A$.

Now, consider again the restriction of $a$ to $M_i \otimes A$ and let us denote it by $\tilde{a}$. Notice that $M_i \otimes A$ has the structure of a right $A$-module (where $A$ acts by multiplication on itself). Then $\tilde{a} : M_i \otimes A \rightarrow <M_i>$ can be seen as a map of $A$-modules, and its kernel $ker(\tilde{a})$ is again an object in $\text{Mod}_{\mathcal{C}}(A)$ such that $<M_i> \simeq( M_i \otimes A) / ker(\tilde{a})$. Let us take a closer look at the kernel $ker(\tilde{a})$ and let $W_j$ be a finite-dimensional $H$-subcomodule of it. As before there exists a map $W_j \otimes A \rightarrow ker(\tilde{a})$. Again, for a directed set $J$ we can consider inclusions $\theta_j^k : W_j \otimes A \rightarrow W_k \otimes A$, whenever $j \leq k$. The colimit of this diagram surjects to $Ker(\tilde{a})$; therefore for fixed $i$, $<M_i> \simeq  M_i \otimes A / \varinjlim W_j \otimes A =  \varinjlim (M_i \otimes A / W_j \otimes A)$. But now the quotients $(M_i \otimes A / W_j \otimes A)$ are compact objects since finite colimits of compact objects are compact. 
We have shown that each $<M_i>$ is a colimit of compact objects, therefore every $M \in \text{Mod}_{\mathcal{C}}(A)$ is a colimit of compact objects.

\end{proof}

\begin{proposition} \label{Mod_A equivalent to coalgebra cat} Let $\mathcal{C}$ be the category of $H$-comodules for a Hopf algebra $H$. Consider the module category $\text{Mod}_{\mathcal{C}}(A)$ for an algebra $A \in \mathcal{C}$. Consider also the functor Ind: $ \text{Mod}_{\mathcal{C}}(A) \rightarrow \mathcal{M}^H$. Then Ind has a right adjoint functor $\tilde{Q}$. Moreover if $\tilde{Q}$ is cocontinuous, then  $\text{Mod}_{\mathcal{C}}(A)$  is equivalent to $\mathcal{M}^C$ for a $k$-coalgebra $C$. 
\end{proposition}

\begin{proof} 

Since $\mathcal{C}$ is the category of right $H$-comodules, it is equipped with a forgetful functor $ \text{Forget} : \mathcal{C} \rightarrow \text{Vect}$ to the category of vector spaces. The forgetful functor is the one that considers the underlying vector space and forgets the comodule structure, and is clearly exact and faithful. Now, the composition $F: = (\text{Forget} \circ \text{Ind}) : \text{Mod}_{\mathcal{C}}(A) \rightarrow \text{Vect}$ is cocontinuous, exact and faithful, being the composition of two such functors. Therefore using Lemma \ref{locally presentable} it can be deduced that $F$ has a right adjoint functor $Q$. In this setting, $Q$ can be computed and it is equal to $Q (V) = \widetilde{\text{HOM}}_{k}(A, V \otimes H)$ (which will be defined below). 
Indeed, we will show that $$ \text{Hom}_{\text{Vect}}(F(M), V) \simeq \text{Hom}_{\text{Mod}_{\mathcal{C}}(A)}(M, \widetilde{\text{HOM}}_{k}(A, V \otimes H)).$$ 
Since $- \otimes H$ is the right adjoint to Forget :  $ \mathcal{M}^H \rightarrow \text{Vect}$ it is enough to show that the functor $\tilde{Q} : \mathcal{M}^H \rightarrow  \text{Mod}_{\mathcal{C}}(A)$ given by $ N \mapsto \widetilde{\text{HOM}}_{k} (A, N)$ is right adjoint to Ind. We recall the definition of HOM which can be found in \cite{Ulb} or \cite{S&O}. Let $A, N$ be $H$-comodules. The space $\text{Hom}_k(A,N)$ is not necessarily an $H$-comodule. However there are maps as follows:
$$ \omega : \text{Hom}_k(A,N) \rightarrow \text{Hom}_k(A, N \otimes H)$$ given by
$$\omega(f)(m) = f(a_{(0)})_{(0)} \otimes f(a_{(0)})_{(1)}S(a_{(1)})$$ and
$$\nu : \text{Hom}_k(A, N) \otimes H \rightarrow \text{Hom}_k(A, N \otimes H)$$ given by
$$\nu(f \otimes h) = f(m) \otimes h.$$
It is proved that $\nu$ is injective and the definition of HOM is given by:
$$\text{HOM}_k(A,N) = \{ f \in \text{Hom}_k(A,N) | \omega(f) \in \text{Im}(\nu) \}.$$
It is then shown in\cite{Ulb} that $\nu^{-1} \circ \omega$ defines a comodule structure on $\text{HOM}_k(A,N)$.

We moreover notice that $\text{HOM}_k(A,N)$ has a natural right $A$-module structure, where the $A$-action is given by $ (f \cdot b)(a)  := f(ba)$. 
We now define $\widetilde{\text{HOM}}_k(A,N)$ as follows :
$$\widetilde{\text{HOM}}_k(A,N) := \{ f \in \text{HOM}_k(A,N) |  (  (f \cdot b)_{(0)} \otimes  (f \cdot b)_{(1)})(a) = (f_{(0)} b_{(0)} \otimes f_{(1)} b_{(1)})(a) \} $$ 

Then $\widetilde{\text{HOM}}_{k}(A, N)$ is an object in $\text{Mod}_{\mathcal{C}}(A)$ by its definition. In particular we defined $\widetilde{\text{HOM}}_k(A,N)$ to consist of the objects in $\text{HOM}_k(A,N)$ that are compatible with the natural right $A$-action. 

We can now proceed to prove the adjunction.

 Let $\Delta : \text{Hom}_{\mathcal{M}^H}(Ind(M), N) \rightarrow \text{Hom}_{\text{Mod}_{\mathcal{C}}(A)}(M , \widetilde{\text{HOM}}_{k}(A, N))$ be defined as follows: 
If $\phi : M \rightarrow N$ is a morphism of $H$-comodules,  $\Delta(\phi)$ is defined to be the morphism : $m \mapsto \tilde{\phi}_{m}$ where $\tilde{\phi}_{m}(a) = \phi(ma)$. 
\begin{enumerate}
\item $\Delta(\phi)(m) \, \in \widetilde{\text{HOM}}_k(A,N)$ for every $m \in M$. Indeed, $$(\tilde{\phi}_m(a_{(0)}))_{(0)} \otimes (\tilde{\phi}_m(a_{(0)}))_{(1)} S(a_{(1)})$$ 
$$= (\phi(m a_{(0)}))_{(0)} \otimes  (\phi(m a_{(0)}))_{(1)}  S(a_{(1)})$$
$$ = (\phi(m a_{(0)}))_{(0)}  \otimes m_{(1)} (a_{(0)})_{(1)} S(a_{(1)}) \, \, (\text{since} \, \, \phi \, \, \text{is an H-comodule map})$$
$$= \phi(m_{(0)} (a_{(0)})_{(0)})  \otimes m_{(1)} (a_{(0)})_{(1)} S(a_{(1)})$$
$$ = (\phi(m_{(0)} a_{(0)}))  \otimes m_{(1)} \epsilon(a_{(1)})$$
$$= (\phi( m_{(0)} a)) \otimes m_{(1)}$$
$$= (\tilde{\phi}_m(a))_{(0)} \otimes m_{(1)}.$$
Since $m_{(1)}$ is stable in $\tilde{\phi}_m$ it can be concluded that $\Delta(\phi)(m)$ lies in $\text{HOM}_k(A,N)$.
Moreover, 
$$ (  (\tilde{\phi}_m \cdot b)_{(0)} \otimes  (\tilde{\phi}_m \cdot b)_{(1)})(a)$$
$$= ((\tilde{\phi}_m \cdot b) (a_{(0)}))_{(0)} \otimes ((\tilde{\phi}_m \cdot b)(a_{(0)}))_{(1)} S(a_{(1)})$$ 
$$= (\phi(m b a_{(0)}))_{(0)} \otimes  (\phi(m b a_{(0)}))_{(1)}  S(a_{(1)})$$
$$ = (\phi(m b a_{(0)}))_{(0)}  \otimes m_{(1)} b_{(1)} (a_{(0)})_{(1)} S(a_{(1)}) $$
$$=\phi(m_{(0)} b_{(0)} (a_{(0)})_{(0)})  \otimes m_{(1)} b_{(1)} (a_{(0)})_{(1)} S(a_{(1)}) $$
$$ = \phi(m_{(0)} b_{(0)} a_{(0)})  \otimes m_{(1)} b_{(1)} \epsilon(a_{(1)})$$
$$= \phi( m_{(0)} b_{(0)} a) \otimes m_{(1)} b_{(1)}$$
$$= (\tilde{\phi}_m(a))_{(0)} b_{(0)} \otimes m_{(1)} b_{(1)}$$
$$ = ((\tilde{\phi}_m)_{(0)} b_{(0)} \otimes (\tilde{\phi}_m)_{(1)} b_{(1)})(a).$$ 

Therefore $\Delta(\phi)(m)$ lies in $\widetilde{HOM}_k(A,N)$.

\item $\Delta(\phi) : M \rightarrow \widetilde{\text{HOM}}_{k}(A, N) $ is $A$-linear. Indeed, $mb \rightarrow \tilde{\phi}_{(mb)}$ and $\tilde{\phi}_{(mb)}(a)= \phi( mba) = (\tilde{\phi}_{m} \cdot b)(a)$.

\item $\Delta(\phi)$ is $H$-colinear. It is enough to show that $\tilde{\phi}_{m_{(0)}}(a) \otimes m_{(1)} = (\tilde{\phi}_{m} (a_{(0)}))_{(0)} \otimes (\tilde{\phi}_{m} (a_{(0)}))_{(1)} S(a_{(a)})$. But by 1 above we know that the right hand side is equal to $ (\phi( m a)) _{(0)} \otimes m_{(1)} = (\tilde{\phi}_m(a))_{(0)} \otimes m_{(1)}$. Since moreover $\phi(m)_{(0)} = \phi(m_{(0)})$, $\phi$ being a morphism of $H$-comodules, the result follows. 

\end{enumerate}

Finally, it is easy to see that $\Delta$ is one-to-one. Moreover, in the opposite direction, given an element in $\text{Hom}_{\text{Mod}_{\mathcal{C}}(A)}(M , \widetilde{\text{HOM}}_{k}(A, N))$ we define an object in $\text{Hom}_{\mathcal{M}^H}(Ind(M), N)$ by $\phi(m) := \tilde{\phi}_m(1_A)$. This is a well-defined map since the morphism $M \rightarrow \widetilde{\text{HOM}}_k(A,N)$ is $H$-colinear. It is also a one-to-one map (if $\tilde{\phi}_m  \neq \bar{\phi}_m$ then there exists an $a \in A$ such that $\tilde{\phi}_m(a)  \neq \bar{\phi}_m(a)$ but   $\tilde{\phi}_m(a)= \tilde{\phi}_m(1_A) \cdot a$ and the same is true for  $\bar{\phi}_m(a)$ ). This means that there is a bijection between the two $\text{Hom}$ sets and that the adjunction holds.

Finally we notice that if $\tilde{Q}$ is cocontinuous, then $Q$ is also cocontinuous and by Theorem \ref{category equivalent to C comodules} we can conclude that $\text{Mod}_{\mathcal{C}}(A)$ is equivalent to $\mathcal{M}^C$ for a $k$-coalgebra $C$.

\end{proof}

\begin{remark} It would be interesting to find explicit conditions under which this adjoint functor is cocontinuous. In the above proposition it seems that some restrictions on the algebra $A$ would be necessary.

\end{remark}


\section{Module categories corresponding to coideal subalgebras} \label{deytero}

In this section we are going to try to find the necessary conditions that a module category $\mathcal{M}$ over $\mathcal{M}^H$ must satisfy in order to be of the form $\text{Mod}_{\mathcal{M}^H}(A)$ for a coideal subalgebra $A$ of the Hopf algebra $H$.

\subsection{Background material}
We start by recalling some background material concerning categories of comodules. The basic reference in this section is \cite{Morita for comodules}. 

\begin{Def} A comodule $X$ in $\mathcal{M}^D$ for a coalgebra $D$  is \textbf{quasi-finite} if $\text{Hom}_{\mathcal{M}^D}(N, X)$ is finite-dimensional for all finite-dimensional comodules $N$. 

\end{Def}

\begin{proposition} [{\cite{Morita for comodules}, 1.3}] For a comodule $X \in \mathcal{M}^D$ , $X$ is quasi-finite if and only if the functor $ \text{Vect} \rightarrow \mathcal{M}^D$, $W \mapsto W \otimes X$ has a left adjoint.

\end{proposition}

When $X$ is quasi-finite, the left adjoint of the functor $W \mapsto W \otimes X$ is called the Cohom functor. It is denoted by $h_D$ and is written as $ Y \mapsto h_D(X, Y)$ for a $D$-comodule $Y$.  This functor has a behaviour similar to the behaviour of the $\text{Hom}$ functor for algebras. 
If now $X$ is a $(\Gamma, D)$-bicomodule and it is also quasi-finite, then $h_D(X, -)$ has the structure of a right $\Gamma$-comodule and the following proposition holds:

\begin{proposition} [{\cite{Morita for comodules}, 1.10}] For a $(\Gamma, D)$- bicomodule $X$ the following are equivalent:
\begin{enumerate}
\item $X$ is quasi-finite.
\item The functor $\mathcal{M}^{\Gamma} \rightarrow \mathcal{M}^D$, $Z \mapsto Z \underset{\Gamma}{\Box} X$ has a left adjoint.
\end{enumerate} 
In this case the left adjoint is the Cohom functor $h_D$.

\begin{remark} The set $\text{Coend}_D(M) = h_D(M, M)$ has the structure of a
coalgebra. Thus $M$ becomes a $(\text{Coend}_D(M), D)$-bicomodule (see also \cite{Morita for comodules}) . 

\end{remark}

\end{proposition}

The following proposition explains how ``nice'' functors between categories of comodules look and will be used heavily in the sequel. 

\begin{proposition} [{\cite{Morita for comodules}, 2.1}] Let $F: \mathcal{M}^{\Gamma} \rightarrow \mathcal{M}^D$ be a $k$-linear functor. If $F$ is left exact and preserves direct sums, then there exists a $(\Gamma, D)$-bicomodule $P$ such that $F(Z) = Z \underset{\Gamma}{\Box} P$ as a functor of $Z \, \in \mathcal{M}^{\Gamma}$.

\end{proposition}

We are now in a position to state the Morita equivalence theorem for categories of comodules.

\begin{Def} A set of \textbf{pre-equivalence data} $(\Gamma, D, _{\Gamma}{P}_D , _D{Q}_{\Gamma}, f, g)$ consists of coalgebras $\Gamma$ and $D$, bicomodules $_{\Gamma}{P}_D$ and $_D{Q}_{\Gamma}$ and bicolinear maps $f: \Gamma \rightarrow P \underset{D}{\Box} Q$ and $g: D \rightarrow Q \underset{\Gamma}{\Box} P$ making the following diagrams commute: 

$$
\text{ \xymatrixcolsep{5pc}
\xymatrix{ {P} \ar[r]^{\simeq} \ar[d]_{\simeq} &{P \underset{D}{\Box} D} \ar[d]^{id \Box g} &{} &{Q} \ar[r]^{\simeq} \ar[d]_{\simeq} &{Q \underset{\Gamma}{\Box} \Gamma} \ar[d]_{id \Box f}\\
{\Gamma \underset{\Gamma}{\Box} P} \ar[r]_{f \Box id} &{P \underset{D}{\Box} Q \underset{\Gamma}{\Box} P} &{} &{D \underset{D}{\Box}Q} \ar[r]_{g \Box id} &{Q \underset{\Gamma}{\Box} P \underset{D}{\Box} Q}  }}$$

\end{Def}

If $f$ and $g$ are isomorphisms then $\Gamma$ and $D$ are said to be Morita-Takeuchi equivalent and their categories of comodules are equivalent. The functors of the equivalence are given by $ - \underset{\Gamma}{\Box} P : \mathcal{M}^{\Gamma} \rightarrow \mathcal{M}^D$ and $- \underset{D}{\Box} Q : \mathcal{M}^D \rightarrow \mathcal{M}^{\Gamma}$. \newline

\subsection{Main theorem}

We are now in the position to prove the second main theorem of the categorical characterization. 
We start with the following lemma. This result can be found in \cite{DM} (after Example 2.15) where a dual version of this theorem is proved. 

\begin{lemma}\label{induced coalgebra map} Let $H, B$ be coalgebras and $\mathcal{M}^H, \mathcal{M}^B$ the corresponding categories of comodules. Assume that there exists a functor $ \Psi : \mathcal{M}^H \rightarrow \mathcal{M}^B$ which carries the forgetful functor to the forgetful functor. Then $\Psi $  is induced by a unique map of coalgebras $\psi: H \rightarrow B$.
\end{lemma}

Now by \cite{Morita for comodules} we know that in a category of comodules $\mathcal{M}^H$, a comodule $Y$ is called quasi-finite if the functor $ \text{Vect} \rightarrow \mathcal{M}^H$ given by $V \mapsto V \otimes Y$ has a left adjoint denoted by $h_Y$ . Therefore $H$ is itself quasi-finite in $\mathcal{M}^H$ and in this case $h_H$ is the forgetful functor. Moreover, it is also proved in \cite{Morita for comodules} (Proposition 1.10), that when $H$ is a $(B, H)$- bicomodule for a coalgebra $B$ then $h_H$ factors through the category $\mathcal{M}^B$ and the factorization map is left adjoint to the cotensor functor $- \underset{B}{\Box} H : \mathcal{M}^B \rightarrow \mathcal{M}^H$. 

This leads to the following:

\begin{corollary}\label{Psi is adjoint to cotensor} Let $H, B$ be coalgebras and assume that there exists a functor  $\Psi : \mathcal{M}^H \rightarrow \mathcal{M}^B$ carrying the forgetful functor to the forgetful functor. Then $\Psi$ is left adjoint to the cotensor functor $- \underset{B}{\Box} H : \mathcal{M}^B \rightarrow \mathcal{M}^H$. 
\end{corollary}  

\begin{proof}
 Indeed,  as we saw in \ref{induced coalgebra map} $\Psi$ is induced by a unique coalgebra map $\psi : H \rightarrow B$.

On the other hand we notice also the following: Whenever we have a map of coalgebras $\psi :H \rightarrow B$, this induces a functor $\psi^* : \mathcal{M}^H \rightarrow \mathcal{M}^B$ as follows:
$$ \text{\xymatrixcolsep{5pc} \xymatrix{ {M} \ar[r]^{\rho_M} & {M \otimes H} \ar[r]^{id \otimes \psi} &{M \otimes B}  \\ {} & {} & {} }}$$
This functor $\psi^*$ obviously commutes with the two forgetful functors to Vect. Therefore we can identify $\Psi$ with $\psi^*$. 
Moreover,  $\psi^*$ is known to be left adjoint to the cotensor functor $- \underset{B}{\Box} H : \mathcal{M}^B \rightarrow \mathcal{M}^H$ (a proof of this appears in \cite{BW} (22.12)), and therefore $\Psi$ is also left adjoint to the cotensor functor.

\end{proof}

\begin{lemma}\label{fc yields surjection}
Assume that the coalgebra map $\psi: H \rightarrow B$ makes $H$ left faithfully coflat over $B$. Then $\psi$ is a surjective map. 

\end{lemma}

\begin{proof}
We first notice that from \cite{NT} (Theorem 3.1) we know that epimorphisms in the category of coalgebras are surjective maps. 
 
The coalgebra map $\psi: H \rightarrow B$ yields the map $ \psi \Box id : H \underset{B}{\Box} H \rightarrow B \underset{B}{\Box} H \simeq H$ of $H$-comodules. By  \cite{BW}, 11.8 we know that the image of $\Delta_H$ is contained in $H \underset{B}{\Box} H$ and that the composition of the maps 
$$ \text{\xymatrixcolsep{5pc} \xymatrix{ {H} \ar[r]^{\Delta_H} & {H \underset{B}{\Box} H} \ar[r]^{\psi \Box id} &{B \underset{B}{\Box} H \simeq H}  \\ {} & {} & {} }}$$
yields the identity. Therefore, $\psi \Box id $ is a surjective map. Since the cotensor functor is exact and faithful it reflects epimorphisms, and this means that $\psi$ is surjective.

\end{proof}

It remains to see under which conditions the map $\psi$ is also left $H$-linear.

\begin{lemma} \label{tensor product of abelian cats}
Let $\mathcal{M}^B$ be a module category over $\mathcal{M}^H$. Suppose that the action of $\mathcal{M}^H$ on $\mathcal{M}^B$ commutes with the two forgetful functors to Vect, i.e. suppose that $ X \otimes M$ is mapped to a $B$-comodule whose underlying vector space is equal to $X \otimes M$. Then the functor $\otimes : \mathcal{M}^H \times \mathcal{M}^B \rightarrow \mathcal{M}^B$ defining the module category $\mathcal{M}^B$ comes from a coalgebra map $\sigma : H \otimes B \rightarrow B$ which endows $B$ with the structure of a left $H$-module. 

\end{lemma}

\begin{proof}
By its definition a module category $\mathcal{M}$ over a monoidal category $\mathcal{C}$ is a category together with an exact  bifunctor $ \otimes : \mathcal{C} \times \mathcal{M} \rightarrow \mathcal{M}$. In particular $\otimes: \mathcal{C} \times \mathcal{M} \rightarrow \mathcal{M}$ is a functor from the Deligne tensor product of abelian categories  $\mathcal{C} \boxtimes \mathcal{M}$ to $\mathcal{M}$. Recall here that the tensor product of abelian categories $\mathcal{A} \boxtimes \mathcal{B}$ is defined by the following requirement:
That for each abelian category $\mathcal{D}$, the category of right exact functors $\mathcal{A} \boxtimes \mathcal{B} \rightarrow \mathcal{D}$ is equivalent to the category of functors $\mathcal{A} \times \mathcal{B} \rightarrow \mathcal{D} $ that are right exact in each variable (see also \cite{Deligne}). 
Consider now the case where $\mathcal{M}^B$  is a module category over $\mathcal{M}^H$. We first notice that $\mathcal{M}^H = \varinjlim{\mathcal{M}^{H_i}}$ where $H = \varinjlim H_i$ as a coalgebra. Indeed, every object $V \in \mathcal{M}^H$ is the colimit of its finite-dimensional vector spaces, and each such finite-dimensional vector space is a comodule over a finite-dimensional subcoalgebra of $H$. The same is true for the coalgebra $B = \varinjlim B_j$ and its category of comodules $\mathcal{M}^B = \varinjlim{\mathcal{M}^{B_j}}$. Therefore it can be concluded that $ \mathcal{M}^H \boxtimes \mathcal{M}^B = \varinjlim{\mathcal{M}^{H_i}} \boxtimes \varinjlim{\mathcal{M}^{B_j}} = \varinjlim \varinjlim (\mathcal{M}^{H_i} \boxtimes \mathcal{M}^{B_j})$ (see also \cite{Deligne}, 5.1 for this property of the tensor product of abelian categories).
Now, it is known again by \cite{Deligne} that $\mathcal{M}^{H_i} \boxtimes \mathcal{M}^{B_j} \simeq \mathcal{M}^{H_i \otimes B_j}$ for finite-dimensional coalgebras. 
On the other hand, $\mathcal{M}^{H \otimes B} = \varinjlim \varinjlim \mathcal{M}^{H_i \otimes B_j}$ since every finite-dimensional subcoalgebra of $H \otimes B$ is of the form $H_i \otimes B_j$ for finite-dimensional subcoalgebras $H_i, B_j$ of $H$ and $B$ respectively. Therefore $ \mathcal{M}^H \boxtimes \mathcal{M}^B \simeq \mathcal{M}^{H \otimes B} $. This means that the exact bifunctor required from the definition of $\mathcal{M}^B$ as a module category over $\mathcal{M}^H$ is actually a functor from $\mathcal{M}^{H \otimes B} \rightarrow \mathcal{M}^B$. Moreover, this functor commutes with the forgetful functors to Vect (by assumption), and therefore, by Lemma \ref{induced coalgebra map} it comes from a coalgebra map $\sigma : H \otimes B \rightarrow B$. By the associativity axiom in the definition of a module category applied to the objects $H \otimes H \otimes B$, it moreover follows that this map $\sigma$ satisfies the properties of a module map, endowing $B$ with the structure of a left $H$-module. 
As a result of the above, the action of an object $X \in \mathcal{M}^H$ on an object $M \in \mathcal{M}^B$ can be interpreted as follows: The object in $\mathcal{M}^B$ obtained by this action is equal to $X \otimes M$ as a vector space, and the $B$-comodule structure is given by the following: \bigskip

\makebox[\textwidth][c]{
$$ \text{\xymatrixcolsep{5pc} \xymatrix{ {X \otimes M} \ar[r]^{\rho_H \otimes \rho_B } & {X \otimes H \otimes M \otimes B } \ar[r]^{id \otimes \text{flip} \otimes id} &{X \otimes M \otimes H \otimes B} \ar[r]^{id \otimes id \otimes \sigma} &{ X \otimes M \otimes B}  \\ {} & {} & {} &{}  }}$$}

\end{proof}

\begin{theorem} \label{theorem 2}
Let $\mathcal{C}$ be the category of $H$-comodules for a Hopf algebra $H$ with bijective antipode. Let $\mathcal{M}^B$ be a module category over $\mathcal{C}$ and assume that there exists a functor of module categories $\Psi : \mathcal{C} \rightarrow \mathcal{M}^B$ which commutes with the two forgetful functors and which has a right adjoint functor $\Omega$ that is exact and faithful. Then $\mathcal{M}^B \simeq \text{Mod}_{\mathcal{C}}(A)$ for a coideal subalgebra $A$ of $H$, and $H$ is faithfully flat over $A$.


\end{theorem}

\begin{proof}

Notice that by Lemma \ref{induced coalgebra map} and Corollary \ref{Psi is adjoint to cotensor} we can deduce that there exists a map of coalgebras $\psi : H \rightarrow B$ and that $\Psi = \psi^*$ is left adjoint to the cotensor functor $ - \underset{B}{\Box} H$. But $\Psi$ is also left adjoint to $\Omega$. Since $\Omega$ is exact and faithful, $- \underset{B}{\Box} H$ must also be exact and faithful. But then by Lemma \ref{fc yields surjection} we know that the coalgebra map $\psi : H \rightarrow B$ is a surjection. 
Now, we can consider the right coideal $^{\text{co}B}H$ in $H$ defined as follows: $ ^{\text{co}B}H = \{ h \in H | \psi(h_{(1)}) \otimes h_{(2)} = \psi(1) \otimes h \}$. We will show that this is also a subalgebra of $H$. 

To do this we will use the arguments of Theorem \ref{theorem 1}.  Indeed, the adjunction $(\Psi, - \underset{B}{\Box} H)$ satisfies all the conditions of \ref{theorem 1} and the composition of the two functors $ \tilde{T} := (-) \underset{B}{\Box} H \circ \Psi$ defines a monad on $\mathcal{C}$.  But now the image of $I\simeq k$ under $\tilde{T}$ is  $\tilde{T}(I) = \Psi(I) \underset{B}{\Box} H $ which is isomorphic to $ ^{\text{co}B}H$. Therefore $^{\text{co}B}H$ is not only a right coideal of $H$ but also an algebra in $\mathcal{M}^H$. Moreover, $\mathcal{M}^B \simeq \text{Mod}_{\mathcal{M}^H}(^{\text{co}B}H) $ as module categories over $\mathcal{M}^H$. 
We need to show that the multiplication map $\mu_I$ on  $ ^{\text{co}B}H$ induced by the monad $\tilde{T}$ coincides with the multiplication on $H$. This is equivalent to showing that the coalgebra map $\psi : H \rightarrow B$ is left $H$-linear. 
By Lemma \ref{tensor product of abelian cats} we know that the action of $\mathcal{M}^H$ is induced by an $H$-module structure $\sigma : H \otimes B \rightarrow B$ on $B$. Moreover, since $\Psi$ is a functor of module categories it follows that $ \Psi( M \otimes V) \simeq M \otimes \Psi(V)$. Now, $M \otimes \Psi(V)$ is the $B$-comodule with underlying vector space $M \otimes V$ and comodule structure given by $\sigma$, therefore $ m \otimes v \mapsto m_{(0)} \otimes v_{(0)} \otimes \sigma(m_{(1)}, \psi(v_{(1)}))$. On the other hand, $\Psi(M \otimes V)$ has comodule structure given by $m \otimes v \mapsto m_{(0)} \otimes v_{(0)} \otimes \psi(m_{(1)} v_{(1)})$. Therefore $\psi$ must be compatible with the $H$-module structure on $B$.

This means that $B$ has the structure of a left $H$-module quotient coalgebra. Moreover $H$ is faithfully coflat over $B$.  Using Corollary \ref{cor1-1} it can be deduced that then $^{\text{co}B}H$ is a right coideal subalgebra of $H$ and that $H$ is faithfully flat over $^{\text{co}B}H$.

\end{proof}

\subsection{The converse of Theorem \ref{theorem 2}} \label{converse}

In this section we prove that a quantum subgroup satisfies the assumptions of Theorem \ref{theorem 2}. This can be seen as a converse statement of the theorem. It also demonstrates how quantum subgroups fit into the categorical picture presented in this section.

\begin{proposition} Let $H$ be a Hopf algebra with bijective antipode and $B$ a left $H$-module quotient coalgebra of $H$.Assume further that $H$ is faithfully coflat over $B$. Let us denote the projection map by $\pi : H \rightarrow B$. Then $\mathcal{M}^B$ is a module category over $\mathcal{M}^H$. Moreover $ \pi^* : \mathcal{M}^H \rightarrow \mathcal{M}^B$ is a functor of module categories which commutes with the two forgetful functors to Vect. Its right adjoint $- \underset{B}{\Box} H : \mathcal{M}^B \rightarrow \mathcal{M}^H$ is exact and faithful and a functor of module categories.
\end{proposition}

\begin{proof}
We have already seen how the coalgebra map $\pi : H \rightarrow B$ induces a functor $\pi^* : \mathcal{M}^H \rightarrow \mathcal{M}^B$ which is left adjoint to the cotensor functor $- \underset{B}{\Box} H : \mathcal{M}^B \rightarrow \mathcal{M}^H$. Moreover $\pi^*$ commutes with the two forgetful functors to Vect by its definition and $- \underset{B}{\Box} H$ is exact and faithful since $H$ is assumed to be faithfully coflat over $B$. It remains to show that both functors are functors of module categories. 

We begin by $\pi^*$. It is enough to show that $\pi^*(Y \otimes N) \simeq Y \hat{\otimes} \pi^*(N)$ for every $Y, N \in \mathcal{M}^H$. We notice that both $\pi^*(Y \otimes N)$ and $Y \hat{\otimes} \pi^*(N)$ are equal to $Y \otimes N$ as vector spaces. We claim that the identiy map is also a morphism of $B$-comodules and therefore that they are isomorphic in $\mathcal{M}^B$. Indeed, the $B$-comodule structure on $\pi^*(Y \otimes N)$ is given by $ y \otimes n \mapsto y_{(0)} \otimes n_{(0)} \otimes \pi(y_{(1)}n_{(1)})$. On the other hand the $B$-comodule structure on $Y \hat{\otimes} \pi^*(N)$ is given by $ y \otimes n \mapsto y_{(0)} \otimes n_{(0)} \otimes y_{(1)}\pi(n_{(1)})$. Since $\pi$ is left $H$-linear by assumption, the isomorphism follows. 

We now proceed with the proof that  $- \underset{B}{\Box} H : \mathcal{M}^B \rightarrow \mathcal{M}^H$ is a functor of module categories. It is enough to show that $X \otimes (M \underset{B}{\Box} H) \simeq (X \hat{\otimes} M) \underset{B}{\Box} H$ in $\mathcal{M}^H$ for every $X \in \mathcal{M}^H$, $M \in \mathcal{M}^B$. We claim that the map $$\gamma : X \otimes (M \underset{B}{\Box} H) \rightarrow (X \hat{\otimes} M) \underset{B}{\Box} H$$ given by $$ x \otimes (m \otimes h) \mapsto (x_{(0)} \otimes m) \otimes x_{(1)}h$$ is an isomorphism with inverse map $\tilde{\gamma}$ given by $$ (x \otimes m) \otimes h \mapsto x_{(0)} \otimes (m \otimes S(x_{(1)})h).$$
We start by showing that if $x \otimes ( m \otimes h)$ is an element of $X \otimes (M \underset{B}{\Box} H)$ then $(x_{(0)} \otimes m) \otimes x_{(1)}h$ is in  $(X \hat{\otimes} M) \underset{B}{\Box} H$. Recall that $$M \underset{B}{\Box} H = \{ m \otimes h | m_{(0)} \otimes m_{(1)} \otimes h = m \otimes \pi(h_{(1)}) \otimes h_{(2)} \}.$$ 
Also note that $$ (X \hat{\otimes} M) \underset{B}{\Box} H = \{ x \otimes m \otimes h | x_{(0)} \otimes m_{(0)} \otimes x_{(1)}m_{(1)} \otimes h = x \otimes m \otimes \pi(h_{(1)}) \otimes h_{(2)} \}.$$ 
Consider $(x_{(0)} \otimes m) \otimes x_{(1)}h$. We need to show that $ (x_{(0)})_{(0)} \otimes m_{(0)} \otimes (x_{(0)})_{(1)} m_{(1)} \otimes x_{(1)} h = x_{(0)} \otimes m \otimes \pi((x_{(1)})_{(1)} h_{(1)}) \otimes (x_{(1)})_{(2)} h_{(2)}$.
Indeed, $$(x_{(0)})_{(0)} \otimes m_{(0)} \otimes (x_{(0)})_{(1)} m_{(1)} \otimes x_{(1)} h $$
$$ = x_{(0)} \otimes m_{(0)} \otimes (x_{(1)})_{(1)} m_{(1)} \otimes (x_{(1)})_{(2)}h$$
$$=  x_{(0)} \otimes m \otimes (x_{(1)})_{(1)} \pi(h_{(1)}) \otimes (x_{(1)})_{(2)}h_{(2)}$$
$$ =  x_{(0)} \otimes m \otimes \pi((x_{(1)})_{(1)} h_{(1)}) \otimes (x_{(1)})_{(2)} h_{(2)}.$$

We further need to check that this map is a morphism of $H$-comodules. The $H$-comodule structure on $x \otimes (m \otimes h)$ is given by $x \otimes m \otimes h \mapsto x_{(0)} \otimes m \otimes h_{(1)} \otimes x_{(1)}h_{(2)}$. On the other hand, the $H$-comodule structure on $(x_{(0)} \otimes m) \otimes x_{(1)}h$ is given by $ x_{(0)} \otimes m \otimes x_{(1)}h \mapsto x_{(0)} \otimes m \otimes (x_{(1)})_{(1)}h_{(1)} \otimes (x_{(1)})_{(2)} h_{(2)}$.  By the above isomorphism $ x_{(0)} \otimes m \otimes h_{(1)} \otimes x_{(1)}h_{(2)} \mapsto (x_{(0)})_{(0)} \otimes m \otimes (x_{(0)})_{(1)} h_{(1)} \otimes x_{(1)}h_{(2)}$. But the last expression is equal to $ x_{(0)} \otimes m \otimes (x_{(1)})_{(1)} h_{(1)} \otimes (x_{(1)})_{(2)}h_{(2)}$. This is exactly what we wanted. 

We will show now that $\tilde{\gamma}$ is also well defined. 
Let $(x \otimes m) \otimes h \in  (X \hat{\otimes} M) \underset{B}{\Box} H$. We want to show that $\tilde{\gamma}((x \otimes m) \otimes h) = x_{(0)} \otimes m \otimes S(x_{(1)})h  \in X \otimes (M \underset{B}{\Box}H)$. For this it is enough to show that $m_{(0)} \otimes m_{(1)} \otimes S(x_{(1)}) h = m \otimes \pi((S(x_{(1)}))_{(1)} h_{(1)}) \otimes (S(x_{(1)}))_{(2)} h_{(2)} = m \otimes S((x_{(1)})_{(2)}) \pi( h_{(1)}) \otimes S((x_{(1)})_{(1)}) h_{(2)} $. 

Since $(x \otimes m) \otimes h \in  (X \hat{\otimes} M) \underset{B}{\Box} H$, we know that: $$ x_{(0)} \otimes m_{(0)} \otimes x_{(1)}m_{(1)} \otimes h = x \otimes m \otimes \pi(h_{(1)}) \otimes h_{(2)}$$ 
$$\Rightarrow x_{(0)} \otimes m_{(0)} \otimes S((x_{(1)})_{(2)}) x_{(1)}m_{(1)} \otimes h = x \otimes m \otimes S((x_{(1)})_{(2)}) \pi(h_{(1)}) \otimes h_{(2)}$$ 
$$\Rightarrow x_{(0)} \otimes m_{(0)} \otimes S((x_{(1)})_{(2)}) x_{(1)}m_{(1)} \otimes S((x_{(1)})_{(1)}) h = x \otimes m \otimes S((x_{(1)})_{(2)}) \pi(h_{(1)}) \otimes S((x_{(1)})_{(1)}) h_{(2)}.$$

We now consider only the left hand side of the above equation and have the following:
$$ x_{(0)} \otimes m_{(0)} \otimes S((x_{(1)})_{(2)}) x_{(1)}m_{(1)} \otimes S((x_{(1)})_{(1)}) h$$
$$ = (x_{(0)})_{(0)} \otimes m_{(0)} \otimes S(x_{(1)}) x_{(1)}m_{(1)} \otimes S((x_{(0)})_{(1)}) h$$
$$= x_{(0)} \otimes m_{(0)} \otimes S(x_{(1)}) (x_{(1)})_{(2)} m_{(1)} \otimes S((x_{(1)})_{(1)}) h$$
$$= x_{(0)} \otimes m_{(0)} \otimes S((x_{(1)})_{(1)} \epsilon((x_{(1)})_{(2)})) (x_{(1)})_{(2)} m_{(1)} \otimes S((x_{(1)})_{(1)}) h$$
$$= x_{(0)} \otimes m_{(0)} \otimes \epsilon(x_{(1)}) \epsilon((x_{(1)})_{(2)})) m_{(1)} \otimes S((x_{(1)})_{(1)}) h$$
$$ = x \otimes m_{(0)} \otimes m_{(1)} \otimes S(x_{(1)})h.$$

This means that $x \otimes m_{(0)} \otimes m_{(1)} \otimes S(x_{(1)})h = x \otimes m \otimes S((x_{(1)})_{(2)}) \pi(h_{(1)}) \otimes S((x_{(1)})_{(1)}) h_{(2)}$ and therefore we can conclude that $\tilde{\gamma}$ is indeed well defined. 

Finally, it is easy to see that $\gamma$  is one-to-one. To show that it is also surjective, we consider the composition with the map $\tilde{\gamma}$ and show that it yields the identity map on $(X \hat{\otimes} M) \underset{B}{\Box} H$. 
Indeed, $$\gamma \circ \tilde{\gamma}(x \otimes m \otimes h)$$
$$ = \gamma(x_{(0)} \otimes m \otimes S(x_{(1)})h)$$
$$ = (x_{(0)})_{(0)} \otimes m \otimes (x_{(0)})_{(1)} S(x_{(1)}) h$$
$$= x_{(0)} \otimes m \otimes (x_{(1)})_{(1)} S((x_{(1)})_{(2)}) h$$
$$= x_{(0)} \otimes m \otimes \epsilon(x_{(1)}) h$$
$$= x \otimes m \otimes h. $$

\end{proof}

\subsection{Morita-Takeuchi equivalence}

Recall the results from \ref{Mod_A equivalent to coalgebra cat} where we assumed that the right adjoint $Q$ is cocontinuous.  We are then in a situation as follows: 

$$
\text{ \xymatrixcolsep{5pc}
\xymatrix{ {\mathcal{M}^H} \ar@/^/[d]^{Forget}  \ar@/^/[r]^{F^T}   & {\text{Mod}_{\mathcal{M}^H}(A)} \ar@/^/[l]_{U^T} \ar[d]^{K} \\
{\text{Vect}} \ar@/^/[u]^{(-) \otimes H} \ar@/^/[r]^{U^G} \ar[ru]^{Q}   & {\mathcal{M}^C} \ar@/^/[l]^{F^G} }} $$

$(F^T, U^T)$ is the monad adjunction for the monad defined in Theorem \ref{theorem 1} using the functors Res and Ind. $F^T$ is then the functor of tensoring an $H$-comodule with $A$, and $U^T$ is exact and faithful. Similarly, $(F^G, U^G)$ is the comonad adjunction from Theorem \ref{category equivalent to C comodules} where $F^G$ is the forgetful functor. $K$ is an equivalence. 

We would like to use Theorem \ref{theorem 2} to deduce that $C$ is actually a quotient left $H$-module coalgebra. However, the functor $F^T \circ K$ does not carry the forgetful functor to the forgetful functor. $F^T(V) = V \otimes A$ and $K(W)= W$. 

We believe however that $\mathcal{M}^C$ is Morita-Takeuchi equivalent to a category of comodules over a coalgebra $B$ such that the composition with the Morita equivalence yields a functor of module categories $\Psi : \mathcal{C} \rightarrow \mathcal{M}^B$ satisfying all conditions of Theorem \ref{theorem 2}. 

This would yield the following diagram:

$$ 
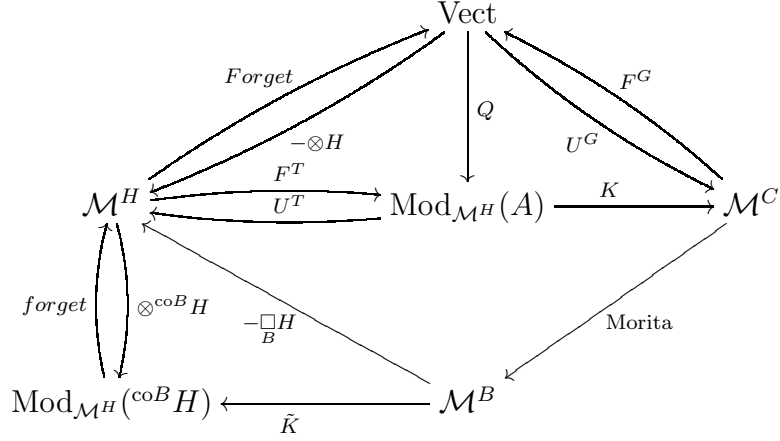
\text{ \xymatrixcolsep{5pc} 
\xymatrix{ {} & {\text{Vect}} \ar@/^/[ldd]^{- \otimes H} \ar@/_/[rdd]_{U^G} \ar[dd]^{Q} &{}\\
{} & {} & {} \\ 
{\mathcal{M}^H} \ar@/^/[ruu]^{Forget} \ar@/^/[dd]^{ \otimes ^{\text{co}B}{H}}  \ar@/^/[r]^{F^T}   & {\text{Mod}_{\mathcal{M}^H}(A)} \ar@/^/[l]_{U^T} \ar[r]^{K} & {\mathcal{M}^C} \ar@/_/[luu]_{F^G} \ar[ldd]^{\text{Morita}} \\
{} & {} & {} \\
{\text{Mod}_{\mathcal{M}^H}(^{\text{co}B}H)} \ar@/^/[uu]^{forget} & {\mathcal{M}^B} \ar[l]^{\tilde{K}} \ar[luu]^{- \underset{B}{\Box} H} & {}   }} $$

$$$$

In this direction, we believe that an extra condition needed to proceed from the results of section \ref{prwto} to the main theorem \ref{theorem 2} of section \ref{deytero} is that $H$ is an object in $\text{Mod}_{\mathcal{M}^H}(A)$.\\ \\


\textbf{Acknowledgements} 
This paper is an edited version of my doctoral thesis. I would like to thank my supervisor Kobi Kremnizer, for giving me this project and introducing me to this beautiful area of research, for his optimism and guidance throughout my work, but also for his constant support and friendship during my entire time in Oxford. I am also grateful to Uli Kr\"ahmer for suggesting to add the section \ref{converse} to this work and also for pointing out which isomorphisms to use in the proof. This work was supported by a scholarship awarded by the Alexander S. Onassis Public Benefit Foundation.




\end{document}